\documentclass{amsart}
\usepackage{amsmath}
\usepackage{amsthm}
\usepackage{amssymb}
\usepackage[alphabetic]{amsrefs}
\usepackage{graphicx}
\usepackage{mathrsfs}
\usepackage{pinlabel}
\usepackage{tikz-cd}
\usepackage{svg}
\usepackage{xypic}
\usepackage{mathtools}
\usepackage{geometry}
\usepackage[utf8]{inputenc}
\usepackage[hidelinks]{hyperref}
\hypersetup{colorlinks,citecolor={red}}  
\usepackage{cleveref}
\usepackage[section]{placeins}
\usepackage{subcaption}

\newtheorem{thm}{Theorem}[section]
\newtheorem*{thm*}{Theorem}

\newtheorem{cor}[thm]{Corollary}
\newtheorem{prop}[thm]{Proposition}
\newtheorem{lem}[thm]{Lemma}

\theoremstyle{definition}
\newtheorem{defn}[thm]{Definition}

\newtheorem{rmk}[thm]{Remark}

\newtheorem{conv}[thm]{Convention}

\newtheorem*{claim}{Claim}

\newcommand{\Z}{\mathbb{Z}}

\newcommand{\R}{\mathbb{R}}

\newcommand{\FF}{\mathcal{F}}
\newcommand{\OO}{\mathcal{O}}
\newcommand{\hyp}{\mathbb{H}}

\newcommand{\Su}{\mathfrak{S}}

\newcommand{\UU}{\mathcal{U}}
\newcommand{\QQ}{\mathcal{Q}}
\newcommand{\MM}{\mathcal{M}}

\newcommand{\wt}{\widetilde}

\title{Depth-one foliations, pseudo-Anosov flows and universal circles}

\author{Junzhi Huang}
\date{\today}

\begin{document}

\begin{abstract}
Given a taut depth-one foliation $\FF$ in a closed atoroidal 3-manifold $M$ transverse to a pseudo-Anosov flow $\phi$ without perfect fits, we show that the universal circle coming from leftmost sections $\Su_\mathrm{left}$ associated to $\FF$, constructed by Thurston and Calegari-Dunfield, is isomorphic to the ideal boundary of the flow space associated to $\phi$ with natural structure maps. As a corollary, we use a theorem of Barthelm\'e-Frankel-Mann to show that there is at most one pseudo-Anosov flow without perfect fits transverse to $\FF$ up to orbit equivalence.
\end{abstract}

\maketitle

\section{Introduction}

There has been an important theme in 3-manifold topology to study the interaction between flows and codimension one foliations in 3-manifolds. The simplest examples of codimension one foliations of a 3-manifold $M$ are fibrations, which are exactly the foliations with all leaves compact. The theory of Thurston norm organizes different ways of fibration of $M$ into a finite number of fibered faces, and there is a one-to-one correspondence between fibered faces and isotopy classes of suspension pseudo-Anosov flows \cite{fried1979fibrations,Th86}. The aim of this paper is to study one of the next simplest classes of foliations, namely depth-one foliations, and their interaction with transverse pseudo-Anosov flows, by comparing the $\pi_1$-actions on $S^1$ that arise in both settings.

A foliation in a closed 3-manifold is called a \textbf{depth-one foliation} if the restriction to the complement of compact leaves is a fibration over the circle. Visually there are a finite number of compact leaves, called depth-zero leaves, and the rest of the leaves (namely the depth-one leaves) are infinite type surfaces spiraling into the depth-zero leaves. One way to construct depth-one foliations is to ``spin" a fibration around an embedded surface (see \cite[Example 4.8]{CalFoliations07}). 

Given a taut depth-one foliation $\FF$ in a closed 3-manifold $M$, a result of Candel \cite{Candel1993} shows that there exists a Riemannian metric on $M$ such that the restrictions on the leaves of $\FF$ are hyperbolic, giving every $\FF$-leaf a standard hyperbolic structure in the sense of \cite{CCHomot} (see also Section \ref{subsec:depth-1}). In particular, there is a natural circle at infinity associated to any leaf of $\FF$. An unpublished construction of Thurston \cite{ThuCircle2}, which was later written down by Calegari-Dunfield \cite{Calegari:2003aa}, produces a circle $\Su_\mathrm{left}$ associated to $\FF$. We will call this circle \textbf{the universal circle from leftmost sections}. The circle $\Su_\mathrm{left}$ is acted on by $\pi_1(M)$, and is equipped with a $\pi_1(M)$-equivariant collection of monotone \textbf{structure maps} $\{U_\lambda\}_{\lambda\in\wt{\FF}}$ to the circles at infinity of all $\wt{\FF}$-leaves, where $\wt{\FF}$ is the lift of $\FF$ to the universal cover $\wt{M}$ of $M$.

In general, there is an axiomatized notion of a universal circle associated to a taut foliation (Definition \ref{def:universal-circle}). The universal circle $\Su_\mathrm{left}$ is a universal circle of $\FF$ in this general sense, but not a canonical one. However, when $\FF$ is a taut depth-one foliation transverse to a pseudo-Anosov flow without perfect fits $\phi$ (see Section \ref{sec:preliminaries} for definitions), we will see that it is possible to relate the $\Su_\mathrm{left}$ to a more natural object, which is the ideal boundary of the flow space of $\phi$.

For a pseudo-Anosov flow $\phi$, the \textbf{flow space} $\OO$ associated to $\phi$ is the space of orbits of the lifted flow $\wt{\phi}$ in $\wt{M}$. It is homeomorphic to $\R^2$ by \cite{BarbotAnosov1995, FenleyAnosov1994, FENLEY2001503}, and there is a compactification $\overline{\OO}=\OO\cup\partial\OO$ given by Fenley \cite{Fenley2005IdealBO}. The ideal boundary $\partial\OO$ is homeomorphic to a circle and the $\pi_1(M)$-action on $\OO$ extends continuously to $\partial\OO$. If $\phi$ has no perfect fits, we show that the shadow of any leaf $\lambda$ of $\wt{\FF}$ provides a natural map $I_\lambda$ from $\partial\OO$ to the circle at infinity of $\lambda$ (see Section \ref{sec:shadows}).

\begin{thm}\label{thm:ideal-boundary}
Let $M$ be a closed atoroidal 3-manifold with a pseudo-Anosov flow $\phi$ without perfect fits, and let $\FF$ be a taut depth-one foliation in $M$ transverse to $\phi$. Then the circle $\partial\OO$, together with the structure maps $\{I_\lambda\}_{\lambda\in\wt{\FF}}$, is a universal circle for $\FF$.
\end{thm}

In work in progress, Landry-Minsky-Taylor show that given a taut foliation $\FF$ almost transverse to a pseudo-Anosov flow $\phi$ in a closed hyperbolic manifold $M$, the boundary of the flow space naturally has the structure of a universal circle for $\FF$, which is a much stronger version of Theorem \ref{thm:ideal-boundary}.

Nevertheless, we show in our setting that the universal circles $\partial\OO$ and $\Su_{\mathrm{left}}$ are isomorphic. More precisely, we have the following theorem.

\begin{thm}\label{thm:main}

Let $M$ be a closed atoroidal 3-manifold with a pseudo-Anosov flow $\phi$ without perfect fits, and let $\FF$ be a taut depth-one foliation in $M$ transverse to $\phi$. Then the $\pi_1(M)$-actions on $\partial\OO$ and on $\Su_\mathrm{left}$ are conjugated by a homeomorphism $T:\Su_\mathrm{left}\to\partial\OO$. Moreover, for any leaf $\lambda$ of $\wt{\FF}$, we have $I_\lambda\circ T=U_\lambda$.

\end{thm}

\begin{cor}\label{cor:orbit-equivalent}
Let $M$ be a closed atoroidal 3-manifold and let $\FF$ be a taut depth one foliation in $M$. Then there is at most one pseudo-Anosov flow without perfect fits transverse to $\FF$ up to orbit equivalence.
\end{cor}

\begin{proof}
Suppose there are two pseudo-Anosov flows without perfect fits $\phi$ and $\varphi$ that are transverse to $\FF$. Since $M$ is atoroidal, both $\phi$ and $\varphi$ are transitive (we say a flow is \textbf{transitive} if it has an orbit that is dense in both positive and negative time) by \cite{MoshDynHomI}. The actions of $\pi_1(M)$ on the ideal circles of the orbit spaces of $\phi$ and $\varphi$ are conjugate by Theorem \ref{thm:main}. By \cite[Theorem 1.5]{Bart2022orbit}, $\phi$ and $\varphi$ are orbit equivalent.
\end{proof}

A conjectural picture of ``pseudo-Anosov packages" is developed in \cite{Calegari2002PromotingEL} by Calegari in the hope that the different structures from taut foliations, laminations, universal circles and pseudo-Anosov flows are organized and compatible in the most natural way, and it is asked to what extent the picture is true.

In particular, given a universal circle $\Su$, he constructs a pair of $\pi_1(M)$-invariant laminations $\Xi^\pm$ on $\Su$. In our case, one can apply the construction to the universal circle $\Su_\mathrm{left}\cong\partial\OO$ and get a pair of laminations on $\partial\OO$. On the other hand, the endpoints of the singular foliations $\FF_\OO^{u/s}$ also induce a pair of laminations $\mathcal{L}_\OO^{u/s}$ on $\partial\OO$ by taking the pairs of endpoints of regular leaves and faces of singular leaves. We partially verify Calegari's picture by showing that $\Xi^\pm$ and $\mathcal{L}_\OO^{u/s}$ coincide.

\begin{thm}\label{thm:invariant-laminations}
In the setting of Theorem \ref{thm:main}, the invariant lamination $\Xi^+$ (resp. $\Xi^-$) on the universal circle $\Su_\mathrm{left}$ equals the induced stable lamination $\mathcal{L}_\OO^{s}$ (resp. $\mathcal{L}_\OO^{u}$) under the isomorphism $T:\Su_\mathrm{left}\to\partial\OO$.
\end{thm}

The organization of the paper is as follows. In Section \ref{sec:preliminaries}, we briefly recall some knowledge about depth-one foliations and pseudo-Anosov flows, together with a description of the leaf space of a depth-one foliation. In Section \ref{sec:shadows}, we summarize the structure of the shadows of leaves of $\FF$ in $\OO$ developed by Cooper-Long-Reid \cite{Cooper:1994aa} and Fenley \cite{Fenley1999823}, and carefully study the infinity of shadows. From there we introduce the restriction maps $I_\lambda$ and a relative version of restriction maps. We start Section \ref{sec:universal-circle} with a brief review of the construction of the universal circle from leftmost sections $\Su_\mathrm{left}$ following \cite{Calegari:2003aa}, and we relate $\Su_\mathrm{left}$ to the universal circle structure of $\partial\OO$ we developed in Section \ref{sec:shadows}. We prove Theorem \ref{thm:main} in Section \ref{sec:homeomorphism} by explicitly constructing the homeomorphism $T$ and proving the desired properties. We conclude with a discussion of invariant laminations and the proof of Theorem \ref{thm:invariant-laminations} in Section \ref{sec:invariant-laminations}.

\subsection*{Acknowledgements} The author is grateful to his advisor Yair Minsky for being inspiring and supportive throughout this project. The author would like to thank Hyungryul Baik, Danny Calegari, Sergio Fenley, Michael Landry and Sam Taylor for helpful comments and conversations. The author is partially supported by NSF grant DMS-2005328.

\section{Preliminaries}\label{sec:preliminaries}

\begin{conv}
We fix a closed Riemannian atoroidal 3-manifold $M$ with the Riemannian metric to be determined later.
For a partition $\Theta$ (eg. a flow, a foliation) of a space $X$ and a point $x\in X$, we let $\Theta(x)$ be the atom of $\Theta$ containing $x$. More generally, if $A$ is a subset of $C$, we use $\Theta(A)$ to denote the saturation of $A$ by $\Theta$-atoms.
\end{conv}

\subsection{Pseudo-Anosov flows}\label{subsec:pA}

We refer to \cite{MoshDynHomI}, \cite{FENLEY2001503} and \cite{ATDynamics2021} for detailed discussions of pseudo-Anosov flows in 3-manifolds. The following definition follows \cite{Fenley2005IdealBO}.

A flow $\phi:M\times\R\to M$ in $M$ is a \textbf{pseudo-Anosov flow} if it has the following properties:

\begin{itemize}
\item each flowline is $C^1$ and not a single point;
\item the tangent line bundle $T\phi$ is continuous;
\item there are a finite number of singular closed orbits, and a pair of 2-dimensional singular foliations $\FF^u$ and $\FF^s$ in $M$ so that:
\begin{itemize}
\item each leaf of $\FF^u$ or $\FF^s$ is a union of $\phi$-orbits;
\item outside of the singular orbits $\FF^u$ and $\FF^s$ are regular foliations whose leaves intersect transversely along $\phi$-orbits;
\item for each singular orbit $\omega$, the leaf of $\FF^u$ or $\FF^s$ containing $\omega$ is homeomorphic to $P_n\times[0,1]/f$ where
\[
P_n=\{re^{\frac{2ki\pi}{n}}|r\geq0, 0\leq k\leq n-1\}\subset\mathbb{C}
\]
is an $n$-prong and $f$ is a homeomorphism from $P_n$ to $P_n$. The orbit $\omega$ is the image of $\{0\}\times [0,1]$ and $n$ is always greater than 2 in our case;
\item orbits in the same $\FF^s$-leaf are forward asymptotic, and orbits in the same $\FF^u$-leaf are backward asymptotic.
\end{itemize} 
\end{itemize}
The singular foliations $\FF^{s}$ and $\FF^{u}$ are called the \textbf{stable foliation} and the \textbf{unstable foliation} of $\phi$ respectively. When the set of singular orbits is empty, $\phi$ is simply an Anosov flow.

Fix a universal cover $\wt{M}$ of $M$ and let $\wt{\phi},\wt{\FF}^u, \wt{\FF}^s$ be the lift of $\phi, \FF^u, \FF^s$ to $\wt{M}$ respectively. The quotient of $\wt{M}$ by $\wt{\phi}$ is the \textbf{flow space} $\OO$, which is homeomorphic to $\R^2$ \cite{BarbotAnosov1995, FenleyAnosov1994, FENLEY2001503}. We orient $\OO$ so that the coorientation coincides with the flow direction, and the pictures are drawn so that the flow is flowing towards the reader. The deck transformation on $\wt{M}$ descends to an orientation preserving $\pi_1(M)$-action on $\OO$, and the singular foliations $\wt{\FF}^u$ and $\wt{\FF}^s$ descend to a pair of $\pi_1(M)$-invariant transverse singular foliations on $\OO$, denoted by $\FF_\OO^u$ and $\FF_\OO^s$. The singular leaves of $\FF_\OO^s$ and $\FF_\OO^u$ are $n$-pronged with $n\geq 3$. The union of two adjacent prongs in a singular leaf is called a \textbf{face}. A leaf of $\FF_\OO^{s}$ or $\FF_\OO^{u}$ is called a periodic leaf if it contains the image of a periodic orbit. In particular, singular leaves are periodic. A point in $\OO$ corresponds to a periodic orbit if and only if it is fixed by a non-trivial element $\gamma\in\pi_1(M)$.

A \textbf{ray} of $\FF_\OO^u$ or $\FF_\OO^s$ is an embedded closed half line contained in a leaf with the interior disjoint from singularities. Two rays $l\in\FF_\OO^s$ and $l'\in\FF_\OO^u$ are said to form a \textbf{perfect fit} if there is an (possibly orientation-reversing) embedding
\[
\iota:[0,1]\times[0,1]-(1,1)\to\OO
\]
mapping horizontal lines to $\FF_\OO^s$ leaves, vertical lines to $\FF_\OO^u$ leaves, $[0,1)\times\{1\}$ to $l'$ and $\{1\}\times[0,1)$ to $l$. We say a pseudo-Anosov flow is \textbf{without perfect fits} if no two rays in $\FF_\OO^s$ and $\FF_\OO^u$ form a perfect fit. The notion of perfect fits is introduced and studied by Fenley in \cite{Fenley1998BranchAnosov,Fenley1999FolGoodGeom}. In particular, we have the following lemma which is an immediate consequence of \cite[Theorem 4.8]{Fenley1999FolGoodGeom}.

\begin{lem}\label{lem:stabilizer}
If $\phi$ has no perfect fits, then any non-trivial element of $\pi_1(M)$ has at most one fixed point in $\OO$.
\end{lem}

Fenley introduced a compactification of $\OO$ in \cite{Fenley2005IdealBO} by building an ideal boundary $\partial\OO$ homeomorphic to $S^1$, and the resulting compactified space $\overline{\OO}=\OO\cup\partial\OO$ is homeomorphic to a closed 2-disk. We orient $\partial \OO$ as the boundary of $\OO$, and the action of $\pi_1(M)$ on $\OO$ extends continuously to an orientation preserving action on $\overline{\OO}$.
Each ray in $\FF_\OO^s$ or $\FF_\OO^u$ has a well-defined end point, and the endpoints of every leaf are distinct. When the flow has no perfect fits, a ray in $\FF_\OO^s$ and a ray in $\FF_\OO^u$ always have distinct endpoints. If we moreover assume that $\phi$ is not conjugate to an Anosov suspension flow (which is automatic when $M$ is atoroidal), the action of $\pi_1(M)$ on $\partial\OO$ is minimal \cite[Main Theorem]{Fenley2005IdealBO}.

\begin{conv}
We assume $\phi$ is a pseudo-Anosov flow without perfect fits in $M$.
\end{conv}

\subsection{End-periodic automorphisms}\label{sec:end-periodic}

We briefly recall the basics of end-periodic automorphisms of infinite type surfaces, which arise naturally in the study of depth-one foliations. Readers are referred to \cite{CANTWELL_CONLON_FENLEY_2021} for a more complete treatment of the theory.

Let $L$ be an infinite type surface without boundary with finitely many non-planar ends. Given a homeomorphism $f:L\to L$, an end $E$ of $L$ is a contracting end of $f$ if there is a neighborhood $U_E$ of $E$ and an integer $n>0$ such that $f^n(U_E)\subsetneq U_E$, and $\bigcap_{k\geq0}f^{nk}(U_E)$ is empty. Such a neighborhood $U_E$ is called a regular neighborhood of $E$. An end is a repelling end of $f$ if it is a contracting end of $f^{-1}$, and a regular neighborhood of $E$ for $f$ is just a regular neighborhood for $f^{-1}$. A homeomorphism $f$ is called \textbf{end-periodic} if each end of $L$ is either contracting or repelling. If $f$ is end-periodic, a multi-curve $\delta$ in $L$ is called an \textbf{$f$-juncture} if $\delta$ is the boundary of a regular neighborhood of an end. If the end is contracting, $\delta$ is called a positive $f$-juncture. Otherwise it is a negative $f$-juncture. An $f$-invariant choice of a positive (resp. negative) $f$-juncture for each contracting (resp. repelling) end is called a \textbf{system of positive (resp. negative) $f$-junctures}. An end-periodic homeomorphism is \textbf{atoroidal} if it does not preserve any essential multi-curve up to isotopy. 

Given an end-periodic homeomorphism $f$, we fix the regular neighborhood $U_E$ for all ends. Let $U^+$ be the union of $U_E$ of contracting ends, and let $U^-$ be the union of those of repelling ends. The \textbf{positive escaping set} $\UU^+$ and the \textbf{negative escaping set} $\UU^-$ is defined as
\[
\UU^\pm=\bigcup_{n\geq 0}f^{\mp n}(U^\pm).
\]
In other words, $\UU^+$ is the set of points whose positive iterations escape to contracting ends, and $\UU^-$ is the set of points whose negative iterations escape to repelling ends. The mapping torus $M_f$ is non-compact, but it is topologically tame and possesses a nice compactification which will be described below.

The mapping torus $M_f$ is the quotient of $L\times\R$ by an automorphism $F$ where $F$ is given by
\begin{equation*}
\begin{aligned}
F:L\times\R&\to L\times\R\\
(x,t)&\mapsto (f^{-1}(x),t+1).
\end{aligned}
\end{equation*}
We attach $\UU^+\times\{+\infty\}$ and $\UU^-\times\{-\infty\}$ to $L\times\R$ to obtain a manifold $N$ with boundary. The transformation $F$ extends to an automorphism of $N$ by setting $F(x,\pm\infty)=(f^{-1}(x),\pm\infty)$. The $\Z$-action on $N$ generated by $F$ is a covering action, and the quotient space $\overline{M_f}$ is a compact 3-manifold with interior $M_f$ and boundary $\partial\overline{M_f}=\partial^+\overline{M_f}\cup\partial^-\overline{M_f}$, where $\partial^\pm\overline{M_f}$ is a (possibly disconnected) closed surface homeomorphic to $\UU^\pm/f$. See \cite{FKLLEndperiodic2023} for a more detailed discussion of the construction. In particular, Lemma 3.3 of \cite{FKLLEndperiodic2023} shows that $\overline{M_f}$ is atoroidal if and only if $f$ is atoroidal.

\subsection{Depth-one foliation}\label{subsec:depth-1}

A foliation $\FF$ in $M$ is a \textbf{depth-one foliation} if $\FF$ has finitely many compact leaves, whose union we denote by $\FF^0$, and $\FF$ restricted to $M-\FF^0$ is a fibration over circle with non-compact fibers. A connected component of $M-\FF^0$ is called a \textbf{fibered region} of $M$. 
Any fibered region $\Omega$ is bounded by leaves in $\FF^0$, and we denote the collection of these leaves by $\partial\Omega$. We say a leaf in $\partial\Omega$ is a \textbf{positive} (resp. \textbf{negative}) \textbf{boundary leaf} of $\Omega$ if it is on the $\phi$-positive (resp. $\phi$-negative) side of $\Omega$. We denote the collection of positive/negative boundary leaves by $\partial^\pm\Omega$. Note it is possible to have a compact leave contained in both $\partial^+\Omega$ and $\partial^-\Omega$. Let $\overline{\Omega}$ be the union of $\Omega$ and $\partial\Omega$.

Let $L$ be a fiber of the fibration $\FF|_\Omega$. The leaf $L$ limits on a compact leaf $\Sigma\subset\partial\Omega$ in the following way \cite{CCPB1981}. Let $N(\Sigma)\cong \Sigma\times[-1,1]$ be a regular neighborhood of $\Sigma$ with $\Sigma$ identified with $\Sigma\times\{0\}$, and assume that $L$ limits on $\Sigma$ on the positive side. If $N(\Sigma)$ is small enough, the intersection of $L$ and $N(\Sigma)$ is an infinite surface spiraling to $\Sigma$ and covering $\Sigma$ with infinite degree. More precisely, up to shrinking $N(\Sigma)$, there is a multi-curve $\delta$ on $\Sigma$ so that $L\cap N(\Sigma)$ is isotopic to an oriented cut-and-paste of $\delta\times (0,1]$ and $\bigcup_{n\geq2}\Sigma\times\{1/n\}$. A fundamental domain for the spiraling is depicted on top of Figure \ref{fig:spiral}, and a schematic picture of the spiraling neighborhood is shown in Figure \ref{fig:juncture}.

We say a foliation is \textbf{taut} if for any leaf of $\FF$, there is a transverse loop intersecting that leaf.

\begin{figure}[h!]
\begin{subfigure}[c]{0.4\textwidth}
\centering
\includegraphics[scale=.17]{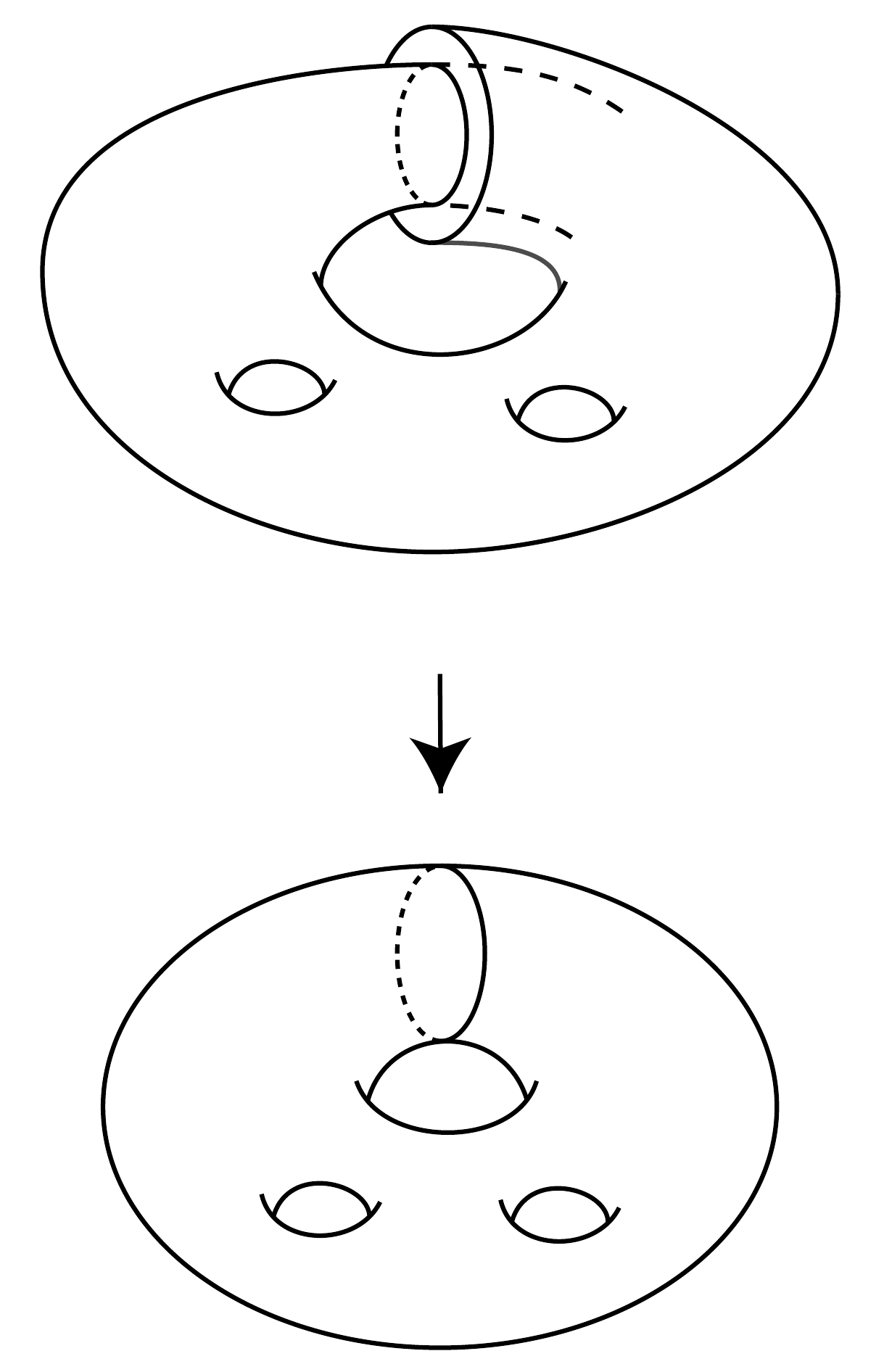}
\labellist
\pinlabel $\Sigma$ at 20 130
\pinlabel $\delta$ at 320 220
\pinlabel $\Sigma$ at 770 110
\pinlabel $\delta$ at 1250 70
\pinlabel $N(\Sigma)$ at 1570 600
\endlabellist
\vfill
\caption{}
\label{fig:spiral}
\end{subfigure}
\begin{subfigure}[c]{0.5\textwidth}
\centering
\includegraphics[scale=.2]{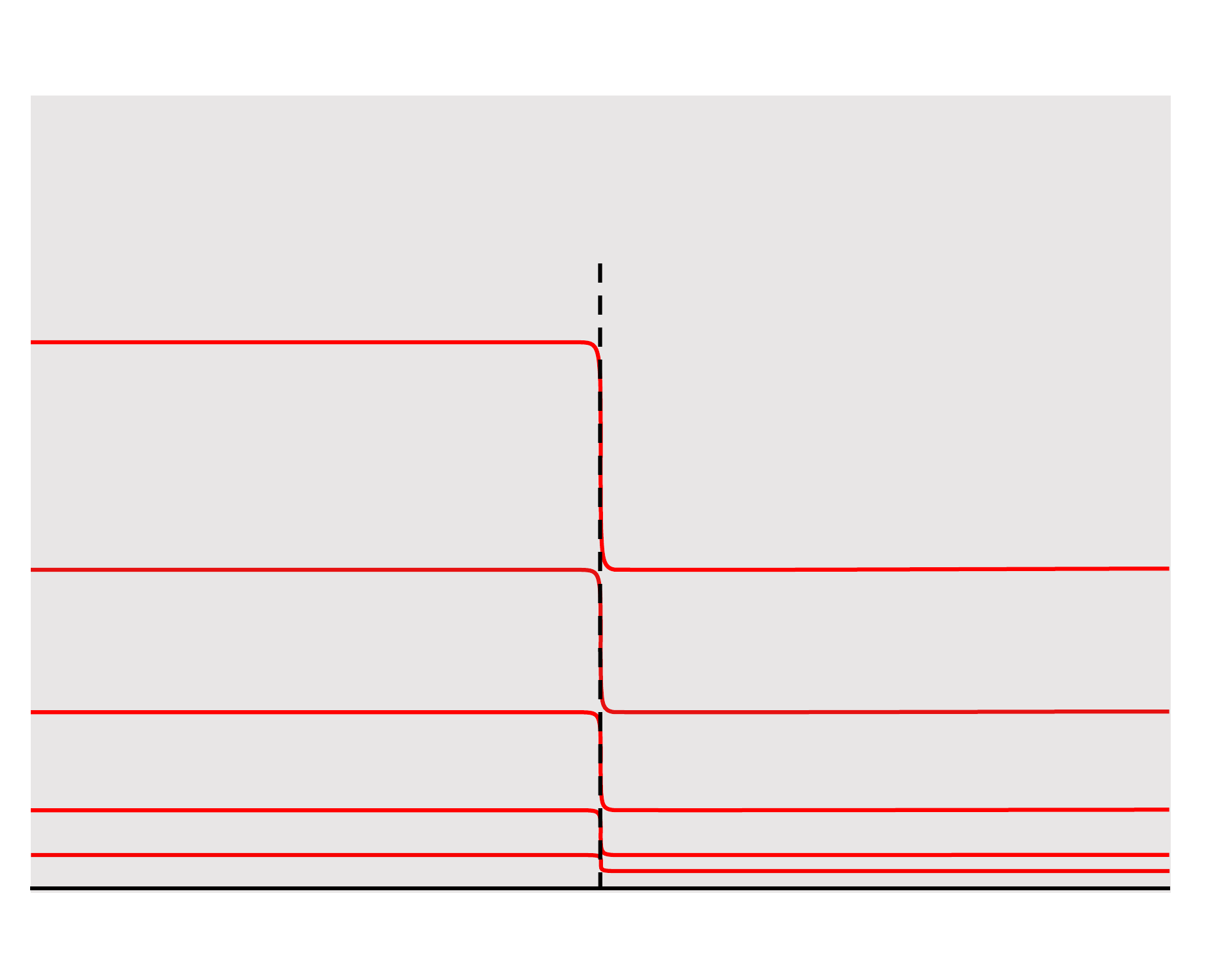}
\vfill
\caption{}
\label{fig:juncture}
\end{subfigure}
\caption{}
\end{figure}

\begin{conv}
We fix a taut depth-one foliation $\FF$ of $M$ and assume that $\FF$ is transverse to $\phi$. By transversality, $\FF$ is coorientable and we take the orientation to be consistent with the flow direction. 
\end{conv}

By transversality of $\FF$ to $\phi$ and by the compactness of $M$, the angle between $T\phi$ and $T\FF$ at any point is uniformly bounded away from zero. In particular, this implies that $\phi|_\Omega$ is a suspension flow of the fibration $\FF|_\Omega$. The flow gives us a way to identify each $\FF$-leaf in $\Omega$ with $L$. Since $\FF$ is taut, every leaf in $\FF^0$ is homologically non-trivial and incompressible by Novikov \cite{NoviFol}. Since $M$ is atoroidal, every compact leaf is a closed hyperbolic surface (the possibility of being a sphere is ruled out by Reeb stability theorem). Therefore, the fundamental domains of the spiraling of $L$ can be chosen to be non-planar. In a spiraling neighborhood of a compact leaf, $\phi$ looks like the product flow. One can see that the first return map induced by $\phi$ is an end-periodic homeomorphism $f:L\to L$ \cite{Fenley1992AsymptoticPO}. Since $M$ is atoroidal, so is $f$. The metric completion of $\Omega$ with respect to the path metric induced by the metric in $M$ gives a compactified mapping torus of $f$.

The metric on $M$ is defined as follows. By \cite{Candel1993}, there is a Riemannian metric on $M$ that restricts to hyperbolic metrics on leaves of $\FF$. We fix such a metric on $M$. A hyperbolic metric on a surface is \textbf{standard} if there is no embedded half-space, following \cite{CCHomot}. The induced hyperbolic metric on any depth-one leaf of $\FF$ is standard, because it has bounded injectivity radius. Indeed, the metric near an end will look more and more like a metric lifted from a hyperbolic closed surface by the nice asymptotic property of depth-one leaves.

Let $\UU^\pm\subset L$ be the positive or negative escaping set of $f$.
By definition, a point $x\in L$ is in $\UU^+$ if and only if the positive ray of $\phi(x)$ hits $\partial^+\Omega$, and it is in $\UU^-$ if and only if the negative ray of $\phi(x)$ hits $\partial^-\Omega$. If $(\overline{\Omega},\partial^-\Omega,\partial^+\Omega)\cong (S\times[0,1],S\times\{0\},S\times\{1\})$ for some closed hyperbolic surface $S$, we say $\Omega$ is a trivial fibered region. This is equivalent to the monodromy $f$ being a pure translation on $L$, and to $\UU^-=\UU^+=L$ \cite[Proposition 4.76]{CANTWELL_CONLON_FENLEY_2021}.

\begin{conv}\label{conv:non-fiber}
We assume that $M$ has no trivial product region. In particular, any compact leaf of $\FF$ is not a fiber. All of our discussions and statements hold with trivial product regions and are readily checked, so we omit the related discussion for simplicity.
\end{conv}

Let $\widetilde{\FF}$ be the lift of $\FF$ in $\widetilde{M}$. The lifted foliation $\widetilde{\FF}$ is a foliation by planes by \cite{NoviFol}. A connected lift of a fibered region of $M$ is called a \textbf{product region} of $\widetilde{M}$. Any product region $\wt{\Omega}$ covering $\Omega$ is homeomorphic to $\wt{L}\times\R$, and we fix a homeomorphism so that $\wt{\Omega}$ is foliated by $\wt{L}\times\{t\}$ and the $\wt{\phi}$-orbits are $\{x\}\times\R$. A lift of a positive/negative boundary leaf of $\Omega$ is a positive/negative boundary leaf of $\wt{\Omega}$, the collection of which is denoted by $\partial^\pm\wt{\Omega}$. Define 
\[
\partial\wt{\Omega}:=\partial^+\wt{\Omega}\cup\partial^-\wt{\Omega}~\text{and}~\overline{\wt{\Omega}}:=\wt{\Omega}\cup\partial\wt{\Omega}
\]

Let $\wt{\UU}^\pm$ be the preimage of $\UU^\pm$ in $\wt{L}$. From the construction of the compactified mapping torus $\overline{\Omega}$ we see that $\overline{\wt{\Omega}}$ is homeomorphic to
\begin{equation}\label{eq:closure-of-pr}
(\wt{L}\times\R)\cup(\wt\UU^+\times\{+\infty\})\cup(\wt\UU^-\times\{-\infty\}).
\end{equation}
The part $\UU^+\times\{+\infty\}$ corresponds to $\partial^+\wt{\Omega}$ and the part $\UU^-\times\{-\infty\}$ corresponds to $\partial^-\wt{\Omega}$.

For any leaf $\mu\in\partial^+\wt\Omega$, there is a component $\wt{\UU}^+_\mu$ of $\wt{\UU}^+$ such that a $\wt\phi$-orbit intersects $\mu$ if and only if it intersects $\wt{L}$ at a point in $\wt{\UU}^+_\mu$. This gives a bijection between leaves in $\partial^+\wt\Omega$ and components of $\wt{\UU}^+$. Similarly, there is a bijection between leaves in $\partial^-\wt\Omega$ and components of $\wt{\UU}^-$.

A leaf of $\widetilde{\FF}$ is called a \textbf{type-0 leaf} if it covers a leaf in $\FF^0$. Otherwise we call it a \textbf{type-1 leaf}. Every type-1 leaf $\mu$ is contained in a unique product region in $\wt{M}$, denoted by $\widetilde{\Omega}(\mu)$. Every type-0 leaf is the negative boundary of a unique production region, and the positive boundary of another different product region. A type-0 leaf $\lambda$ and a type-1 leaf $\mu$ are called \textbf{adjacent} if $\lambda\subset\partial\widetilde{\Omega}(\mu)$. If moreover $\lambda$ is in the positive side of $\mu$, we say $\lambda$ is \textbf{positively adjacent} to $\mu$ or $\mu\lesssim\lambda$; otherwise we say $\lambda$ is \textbf{negatively adjacent} to $\mu$ or $\mu\gtrsim\lambda$.

\begin{figure}[h!]
\centering
\includegraphics[scale=.15]{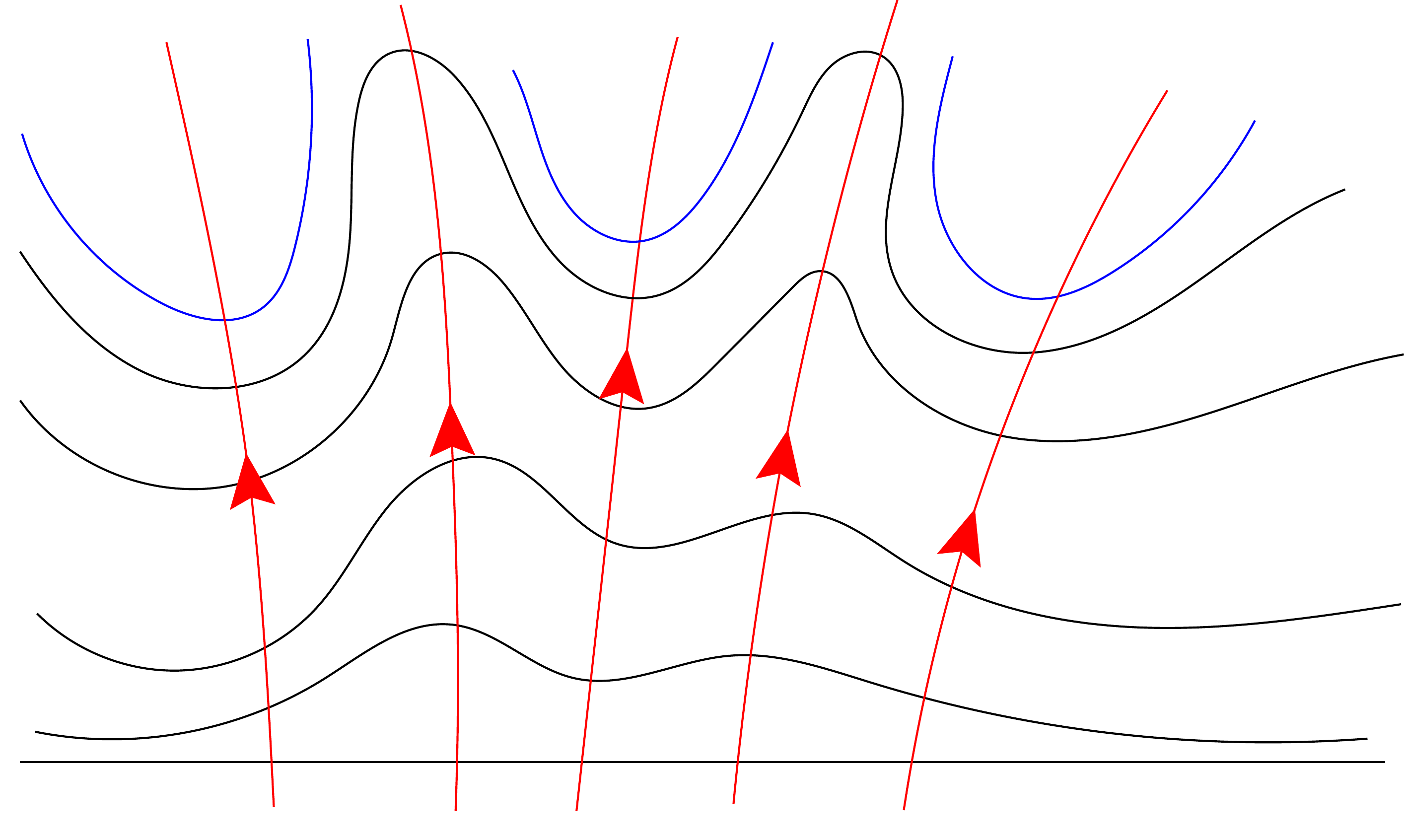}
\caption{Type-1 leaves (black) in a product region limit to type-0 leaves (blue) in the positive boundary and stay transverse to $\wt{\phi}$ (red), creating non-Hausdorff-ness in $\Lambda$}
\label{fig:cataclysm}
\end{figure}

Let $\Lambda$ be the leaf space of $\widetilde{\FF}$ obtained by collapsing each leaf to a point, which is a non-Hausdorff 1-manifold. Since each leaf of $\wt{\FF}$ is properly embedded in $\wt{M}$ and hence separating, $\Lambda$ is simply connected. Each product region $\widetilde{\Omega}$ projects to an oriented open interval in $\Lambda$, with the orientation induced by $\wt{\phi}$. Every such open interval has a countably infinite number of positive endpoints, each corresponding to a component of $\partial^+\wt{\Omega}$. The positive endpoints of the open intervals are non-separated in $\Lambda$ by (\ref{eq:closure-of-pr}). The same is true for negative endpoints. The closures of product regions are glued together along type-0 leaves in the boundary, so each point corresponding to a type-0 leaf is the negative endpoint of exactly one open interval associated to a product region, and the positive endpoint of exactly another different one.

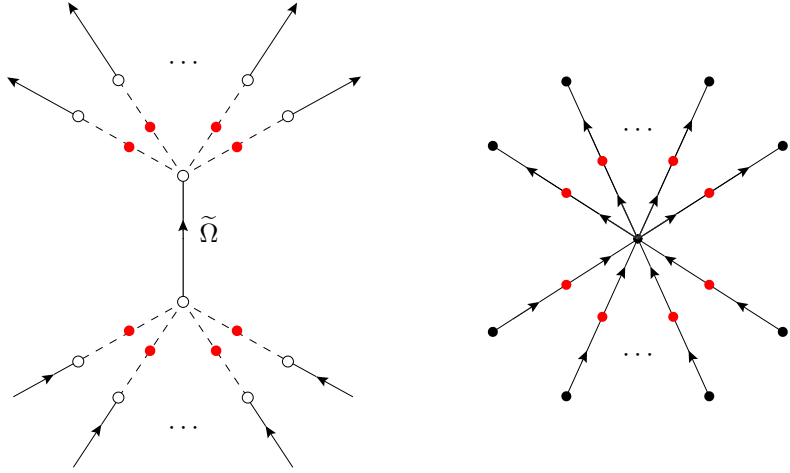
\begin{figure}[h!]
\centering
\labellist
\pinlabel $\wt{\Omega}$ at 320 400
\pinlabel $\cdots$ at 285 650
\pinlabel $\cdots$ at 285 100
\pinlabel $\cdots$ at 970 550
\pinlabel $\cdots$ at 970 210
\endlabellist
\includegraphics[scale=.25]{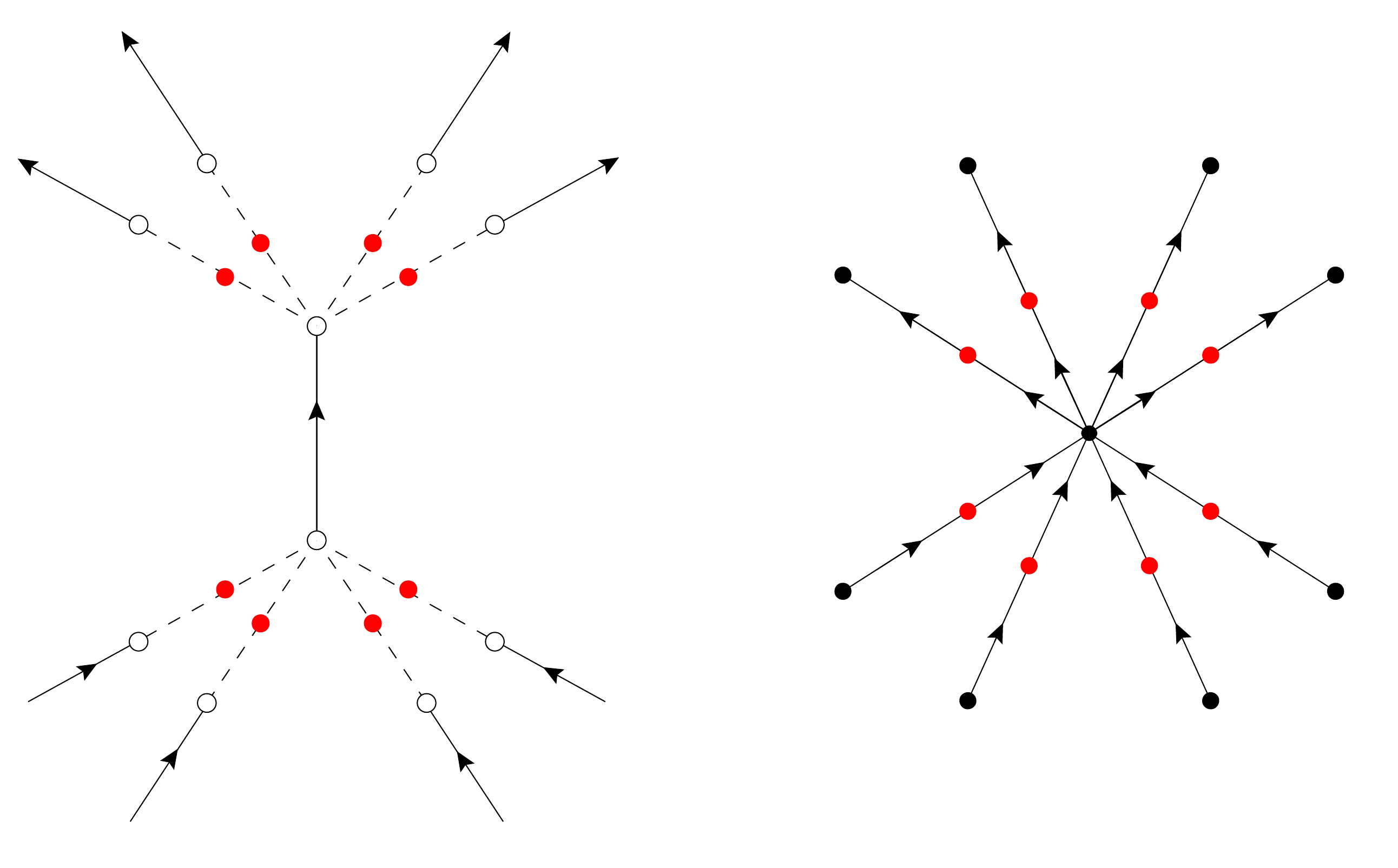}
\caption{Left: a local picture near a product region $\wt{\Omega}$ in $\Lambda$. Right: the corresponding parts in $\Lambda^*$. The red vertices represent type-0 leaves in both pictures, and the arrows indicate the direction of $\wt{\phi}$.}
\label{fig:leaf-space}
\end{figure}

It is sometimes useful to think about the dual graph $\Lambda^*$. The set of vertices are the set of product regions and type-0 leaves. The edges are the pairs $(\Omega,\mu)$ where $\Omega$ is a product region, $\mu$ is a type-0 leaf and $\mu\in\partial\Omega$. The dual graph $\Lambda^*$ is an infinite valence tree with an orientation given by the flow. 

Two leaves of $\widetilde{\FF}$ are called \textbf{comparable} if they can be connected by an oriented path in $\Lambda$. Otherwise, they are \textbf{incomparable}. If $\lambda$ is comparable to $\mu$ and $\mu$ is in the positive side of $\lambda$, we write $\lambda<\mu$; if $\mu$ is in the negative side of $\lambda$, we write $\mu<\lambda$. Similarly, we say two product regions are comparable if there is an oriented path connecting them in $\Lambda^*$, and are incomparable if otherwise.

\subsection{Laminations on $S^1$}
We recall some definitions and constructions of abstract laminations on $S^1$.

Let $\mathrm{Symm}_2(S^1):=S^1\times S^1-\Delta/\sim$ be the space of unordered pairs of distinct points in $S^1$ endowed with the quotient topology, where the relation is given by $(x,y)\sim(y,x)$. Two pairs of points $\{x,y\}$ and $\{z,w\}$ on $S^1$ are said to be unlinked if $z$ and $w$ lie in the same component of $S^1-\{x,y\}$. A lamination on $S^1$ is a closed pairwise unlinked subset of $\mathrm{Symm}_2(S^1)$.

By identifying $S^1$ with $\partial\mathbb{H}^2$, any lamination $\Xi$ on $S^1$ determines a geodesic lamination $\Xi_\mathrm{geod}$ on $\mathbb{H}^2$ by taking the union of geodesics that connect the pairs in $\Xi$. Conversely, any geodesic lamination in $\hyp^2$ gives a lamination on $S^1$ consisting of endpoint pairs of leaves.

Given any subset $A\subset S^1$, the boundary of the convex hull of $\overline{A}$ is a geodesic laminations on $\hyp^2$, which can be viewed as a lamination $\partial\mathrm{CH}(A)$ on $S^1$. Note that the lamination $\partial\mathrm{CH}(A)$ is independent of the choice of the identification $S^1\cong\partial\mathbb{H}^2$.

\section{Infinity of shadows}\label{sec:shadows}

Let $p:\widetilde{M}\to\OO$ be the projection map. For any subset $A$ of $\widetilde{M}$, the image of $A$ under $p$ is called the \textbf{shadow} of $A$. In this section, we will study the shadow of leaves of $\wt\FF$, especially the behavior of the shadow at infinity. The main results are Lemma \ref{lem:c-continuous-extension} and Lemma \ref{lem:nc-continuous-extension}, which hint at a universal circle structure on $\partial\OO$.

Let $\lambda$ be a type-0 leaf of $\widetilde{\FF}$, and let $\Sigma$ be the leaf of $\FF$ it covers. Denote the covering map by $\pi_0$. Since $\lambda$ separates $\wt{M}$ and the flow $\wt{\phi}$ crosses $\lambda$ positively, so each flowline intersects $\lambda$ at most once. It follows that $p$ restricted to $\lambda$ is an orientation-preserving homeomorphism to its image. The map $\pi:=\pi_0\circ (p|_\lambda)^{-1}:p(\lambda)\to \Sigma$ is then a covering map, so we can view $p(\lambda)$ as a universal cover of $\Sigma$. The deck transformation group action on $p(\lambda)$ is simply the action of $\pi_1(M)$ on $\OO$ restricted to a subgroup in the conjugacy class of $\pi_1(\Sigma)$ that stabilizes $p(\lambda)$. We identify this subgroup with $\pi_1(\Sigma)$. By transversality, the intersection of $\FF^{s/u}$ with $\Sigma$ induces a singular foliation on $\Sigma$, denoted by $\FF_\Sigma^{s/u}$. Similarly, let the intersection of $\wt{\FF}^{s/u}$ with $\lambda$ be denoted by $\wt{\FF}_\lambda^{s/u}$. These foliations are related by the relations $p(\wt{\FF}_\lambda^{s/u})=\FF_\OO^{s/u}|_{p(\lambda)}$ and $\pi(\FF_\OO^{s/u}|_{p(\lambda)})=\FF_\Sigma^{s/u}$.

\subsection*{Shadows of type-0 leaves}
Assume $\lambda$ is a type-0 leaf, so $\Sigma$ is an embedded closed surface in $M$ which is not a fiber by Convention \ref{conv:non-fiber}. The shape of the shadow of a non-fibered transverse closed embedded surface is carefully studied in \cite{Cooper:1994aa} when $\phi$ is a pseudo-Anosov suspension flow, later generalized by \cite{Fenley1999823} to general pseudo-Anosov flows. Each component of the topological boundary of $p(\lambda)$ in $\OO$ is either a regular leaf of $\FF_\OO^u$ or $\FF_\OO^s$, or a face of a singular leaf that is regular on the side containing $p(\lambda)$ \cite[Proposition 4.3]{Fenley1999823}. We call a boundary component of $p(\lambda)$ a \textbf{side} of $p(\lambda)$. Consider a side $e$ of $p(\lambda)$, which we assume to be contained in a leaf of $\FF_\OO^s$ for concreteness. We collect some useful facts about the local dynamics at $e$ from \cite{Fenley1999823} in the following proposition.

\begin{prop}[\cite{Fenley1999823}] \label{prop:local-dynamics}
The stabilizer of $e$ in $\pi_1(M)$ is isomorphic to $\Z$ and contained in $\pi_1(\Sigma)$. There is a generator $g_e$ acting on $p(\lambda)$ with the following dynamics (Figure \ref{fig:edge}). 
\begin{itemize}
\item [(1)] The element $g_e$ acts as a contraction on $e$ with a unique fixed point $x_e$.
\item [(2)] The element $g_e$ fixes and expands $l_e:=\FF_\OO^u(x_e)\cap p(\lambda)$, which is the interior of a ray of $\FF_\OO^u$. We have that $l_e$ projects via $\pi$ to a closed leaf $\alpha_e$ of $\FF_\Sigma^u$, whose free homotopy class is represented by $g_e$.
\item [(3)] For any point $x$ other than $x_e$ in $e$, the intersection $l_x:=\FF_\OO^u(x)\cap p(\lambda)$ is connected and it projects via $\pi$ to a non-compact leaf of $\FF_\Sigma^u$ that spirals into $\alpha_e$. If we orient $l_x$ and $l_e$ so that they point towards $e$ and let $\pi(l_x)$ and $\alpha_e$ inherit the orientation, then $\pi(l_x)$ and $\alpha_e$ are asymptotic in the forward direction.
\end{itemize}
\end{prop}

\begin{figure}
\centering
\labellist
\pinlabel $x_e$ at 300 300
\pinlabel $l_e$ at 450 480
\pinlabel $e$ at 160 430
\endlabellist
\includegraphics[scale=.2]{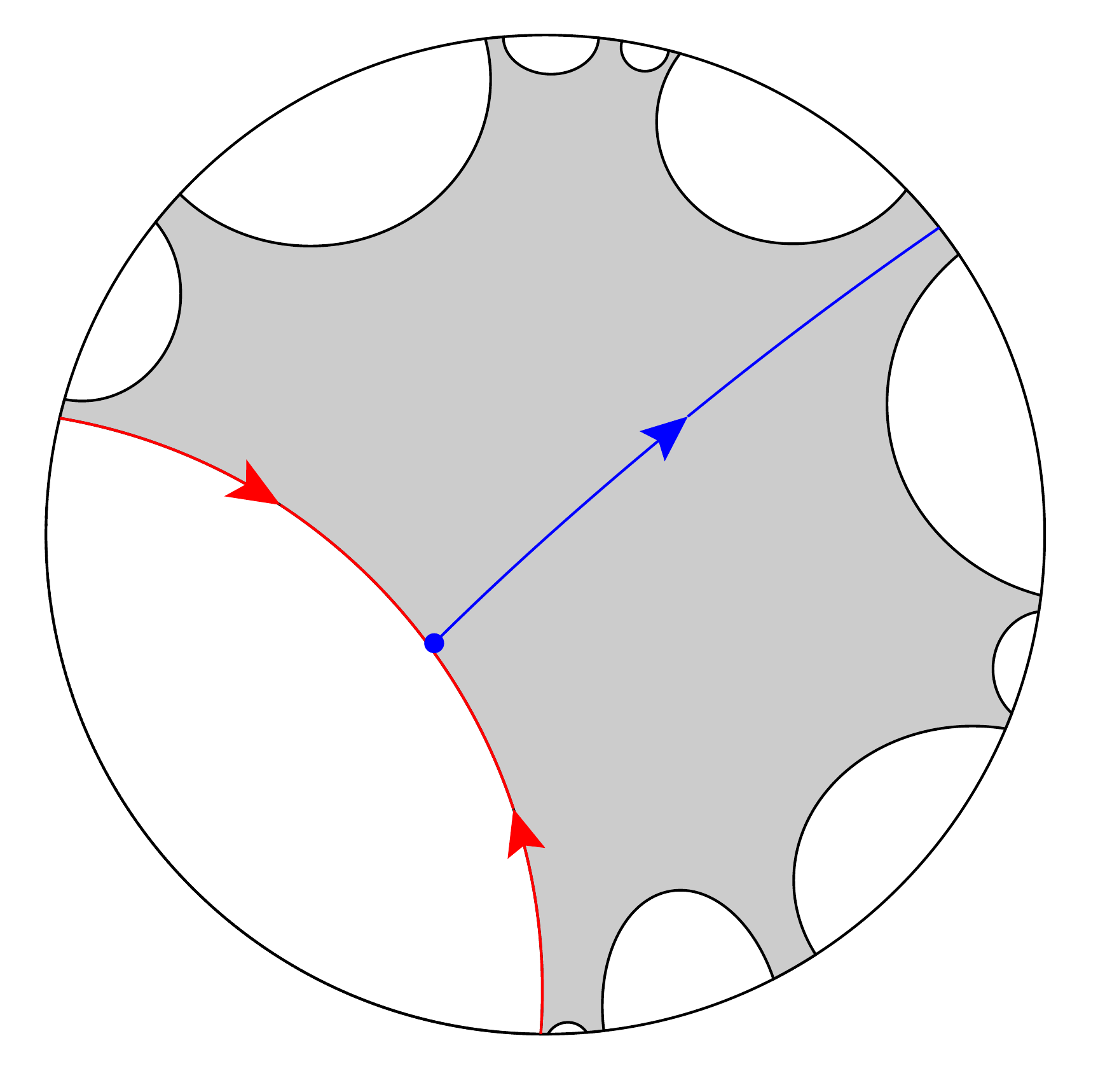}
\caption{}
\label{fig:edge}
\end{figure}

The 3-dimensional picture is the following (Figure \ref{fig:3d}). In $\wt{M}$, the orbit $\sigma_e:=p^{-1}(x_e)$ is a periodic orbit disjoint from $\lambda$. The leaf $p^{-1}(e)\subset\wt{\FF}^s(\sigma_e)$ does not intersect $\lambda$, and $p^{-1}(l_e)\subset\wt{\FF}^u(\sigma_e)$ intersects $\lambda$ transversely in a line $\wt{\alpha}_e$, so that all the $\wt{\phi}$-orbits in $p^{-1}(l_e)$ cross $\lambda$ positively. The line $\wt{\alpha}_e$ covers the simple closed curve $\alpha_e$ in $\Sigma$. The element $g_e$ is a translation that fixes $\sigma_e$ and $\lambda$.

For a side $e$ contained in a leaf of $\FF_\OO^s$, we have a similar picture. From now on, given a side $e$ of $p(\lambda)$ we will continue to use the notations $g_e$, $l_e$, $x_e$ and $\alpha_e$ for the objects described in Proposition \ref{prop:local-dynamics}. In particular, we take $g_e$ to be the generator in the stabilizer of $e$ that contracts $e$.

\begin{figure}
\centering
\labellist
\endlabellist
\includegraphics[height=2in]{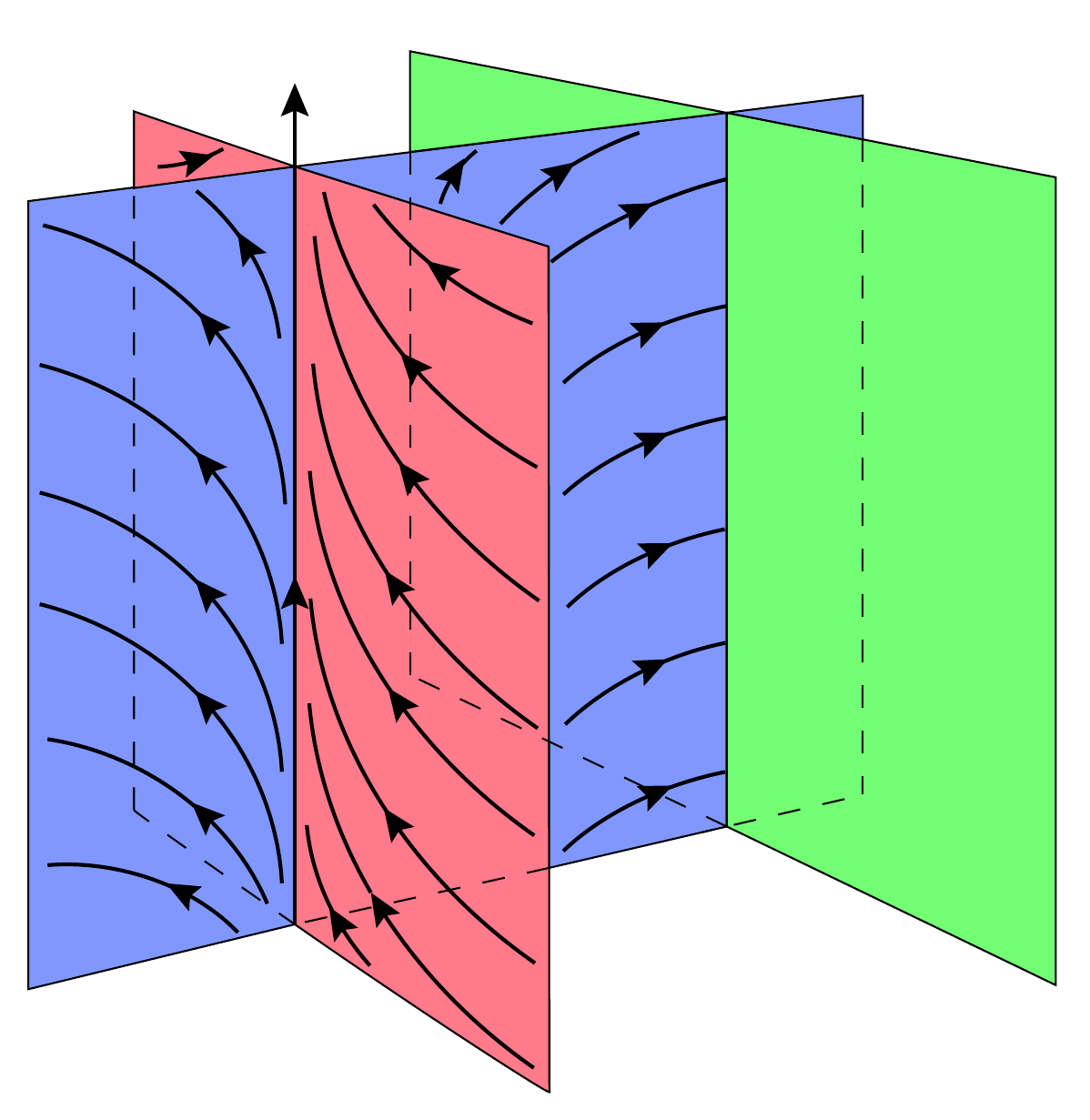}
\caption{The vertical arrowed line represents the periodic orbit $\sigma_e$, the red plane is $\wt\FF^s(\sigma_e)$, the blue plane is $\wt{\FF}^u(\sigma_e)$, and the green plane is $\lambda$.}
\label{fig:3d}
\end{figure}

Let $\partial p(\lambda)$ be the boundary of $p(\lambda)$ \emph{in $\overline{\OO}$}, consisting of sides and ideal boundary points at infinity of $p(\lambda)$. 

\begin{lem}\label{lem:density-of-edges}
The intersection $\partial p(\lambda)\cap\partial\OO$ is nowhere dense.
\end{lem}

\begin{proof}
It is clear that $\partial p(\lambda)\cap\partial\OO$ is a closed subset of $\partial\OO$. Suppose that there is a maximal closed interval $A$ contained in $\partial p(\lambda)\cap\partial\OO$ with endpoints $\eta_1$ and $\eta_2$. Then there is a side $e_1$ of $p(\lambda)$ such that $\eta_1$ is an endpoint of $e_1$. Take $g_{e_1}$, $x_{e_1}$ and $l_{e_1}$ as before. Let $c$ be the endpoint of $l_{e_1}$ at $\partial\OO$.

There is another side $e_2$ of $p(\lambda)$ such that $\eta_2$ is an endpoint of $e_2$. If $c$ does not lie in $A$, then $e_2$ is between $c$ and $\eta_1$ (Figure \ref{fig:density-of-edges}). The action of $g_{e_1}$ fixes $\eta_1$ and $c$, but it cannot fix $e_2$. Otherwise $g_{e_1}$ will have two fixed points on $\OO$, contradicting the assumption that $\phi$ has no perfect fits by Lemma \ref{lem:stabilizer}. Then one of $g_{e_1}^{\pm1}(\eta_2)$ will be in $A$, contradicting our choice of $A$. 

\begin{figure}[h!]
\centering
\labellist
\pinlabel $x_{e_1}$ at 550 250
\pinlabel $\eta_1$ at 270 70
\pinlabel $\eta_2$ at 50 300
\pinlabel $e_1$ at 750 300
\pinlabel $e_2$ at 190 420
\pinlabel $l_{e_1}$ at 350 550
\pinlabel $c$ at 50 650
\pinlabel $A$ at 130 170
\endlabellist
\includegraphics[scale=.18]{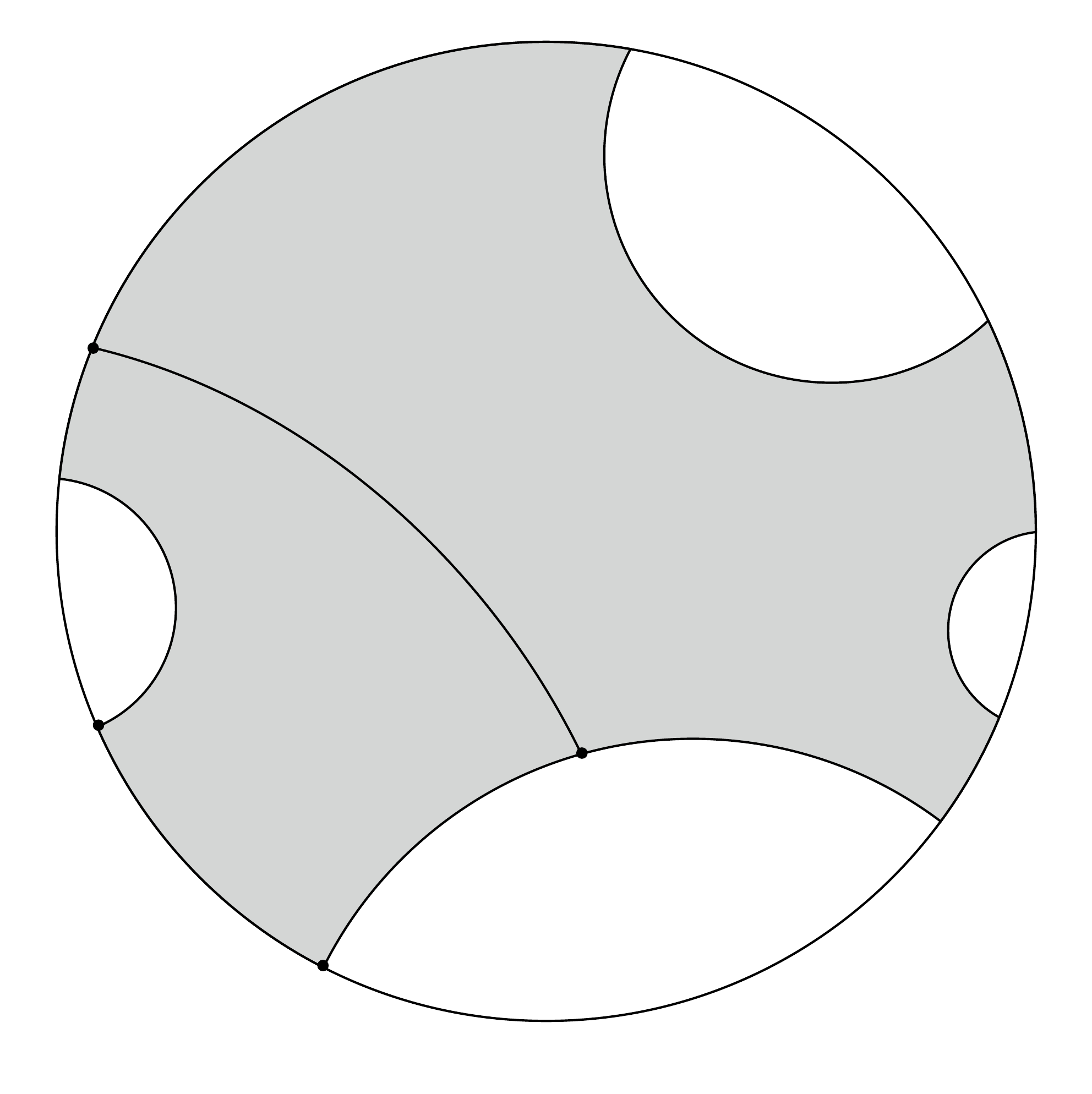}
\caption{}
\label{fig:density-of-edges}
\end{figure}

Hence, $c$ must lie in $A$. Now $l_{e_1}$ divides $p(\lambda)$ into two connected components, one of which contains no side of $p(\lambda)$. In particular, this component contains no $\pi_1(\Sigma)$-translation of $l_{e_1}$. Translated to the hyperbolic plane by $Q_\lambda$, this means there is a simple closed curve $\alpha_{e_1}\in\Sigma$ and a lift $Q_\lambda(l_{e_1})$ of $\alpha_{e_1}$ in $\lambda$ such that there is no other lift on one side of $\wt{\alpha}$. But that is impossible. We conclude that such interval $A$ does not exist.
\end{proof}

The boundary $\partial p(\lambda)$ is homeomorphic to a circle, and $\overline{p(\lambda)}:=\partial p(\lambda)\cup p(\lambda)$ is the closure of $p(\lambda)$ in $\overline{\OO}$. Note that our choice of the metric on $M$ restricts to a hyperbolic metric on $\Sigma$, so $\lambda$ is isometric to $\hyp^2$. We define $\overline{\lambda}$ to be the usual compactified hyperbolic plane with ideal boundary $\partial_\infty\lambda$. The next lemma reveals how the boundary of $p(\lambda)$ is related to $\partial_\infty\lambda$.

Recall that a map $g$ from $S^1$ to $S^1$ is called \textbf{monotone} if the preimage of any point on $S^1$ is contractible.  A \textbf{gap} of $g$ is a maximal closed interval of positive length in $S^1$ that is collapsed to a single point by $g$. The \textbf{core} of $g$ is the complement of the union of the interiors of gaps, denoted by $\mathrm{core}(g)$. We remark that the gaps are sometimes taken to be open intervals in the literature, different from our convention.

\begin{lem}\label{lem:c-continuous-extension}
    Let $\lambda$ be a type-0 leaf of $\wt{\FF}$. Then the homeomorphism $p^{-1}|_{p(\lambda)}: p(\lambda)\to\lambda$ extends continuously to a map $Q_\lambda$ from $\overline{p(\lambda)}$ to $\overline{\lambda}:=\partial_\infty \lambda\cup\lambda$. The map restricted to $\partial{p(\lambda)}$ is a monotone map to $\partial_\infty\lambda$ with core $\overline{p(\lambda)}\cap\partial\OO$.
\end{lem}

\begin{proof}

In the interior of $p(\lambda)$, we define $Q_\lambda=p^{-1}|_{p(\lambda)}$ in the interior of $p(\lambda)$. To extend $Q_\lambda$ to $\partial p(\lambda)$, we will first define the map on the sides of $p(\lambda)$ and then extend it to the entire $\partial p(\lambda)$.

Let $e$ be a side of $p(\lambda)$. Take $g_e\in\mathrm{Stab}(e)$, $x_e\in e$ and $l_e$ as in Proposition \ref{prop:local-dynamics}. As an element of $\pi_1(\Sigma)$, $g_e$ acts as a hyperbolic element on $\overline{\lambda}$ with a contracting fixed point $\partial^-g_e$ and a repelling fixed point $\partial^+g_e$ at infinity. We define $Q_\lambda$ on $e$ as the constant map to $\partial^-g_e$. Different sides are sent to different points in $\partial_\infty(\lambda)$ because of Lemma \ref{lem:stabilizer}.

Since $Q_\lambda(l_e)$ projects to a simple closed curve $\alpha_e$ in $\Sigma$, it is a quasi-geodesic in $\lambda$ with well-defined endpoints in $\partial_\infty\lambda$. Moreover, $\alpha_e$ represents the free homotopy class of $g_e$ up to taking the inverse, by Proposition \ref{prop:local-dynamics}. Therefore, the endpoints of $Q_\lambda(l_e)$ are $\partial^\pm g_e$. By the way we choose $g_e$ (item (2) of Proposition \ref{prop:local-dynamics}), it contracts $l_e$ near $e$. Therefore, if we orient $l_e$ to point towards $e$ and give $Q_\lambda(l_e)$ the induced orientation, the forward endpoint of $Q_\lambda(l_e)$ is $\partial^-g_e$. This shows $Q_\lambda$ is continuous at $x_e$ when restricted to $l_e\cup x_e$. This further assures that $Q_\lambda$ preserves the cyclic order of the boundary leaves, which is a consequence of the fact that $Q_\lambda$ is an orientation preserving homeomorphism in the interior. For example, one can argue as the following (indicated in Figure \ref{fig:cyclic-order}). Consider three sides $e_1,e_2$ and $e_3$ in clockwise order, and take $x_{e_i}$ and $l_{e_i}$ as before. Take a point $y_{e_i}\in l_{e_i}$ and let $l'_{e_i}\subset l_{e_i}$ be the segment between $x_{e_i}$ and $y_{e_i}$. We may assure $l'_{e_i}$ are pairwisely disjoint by taking $y_{e_i}$ close enough to $x_{e_i}$. Take any point $z\in p(\lambda)$ and connect $y_{e_i}$ to $z$ to get an embedded 3-prong $P$ with $z$ as the center and $x_{e_i}$ as the endpoints. The image $Q_\lambda(P)$ is an embedded 3-prong in $\overline{\lambda}$ with well-defined endpoints $Q_\lambda(e_i)$. The cyclic order of the edges of $R$ is preserved under $Q_\lambda$, assuring that $Q_\lambda(e_1), Q_\lambda(e_2)$ and $Q_\lambda(e_3)$ are arranged clockwisely.

\begin{figure}[h!]
\centering
\labellist
\pinlabel $z$ at 340 330
\pinlabel $Q_\lambda(z)$ at 1280 330
\endlabellist
\includegraphics[scale=.2]{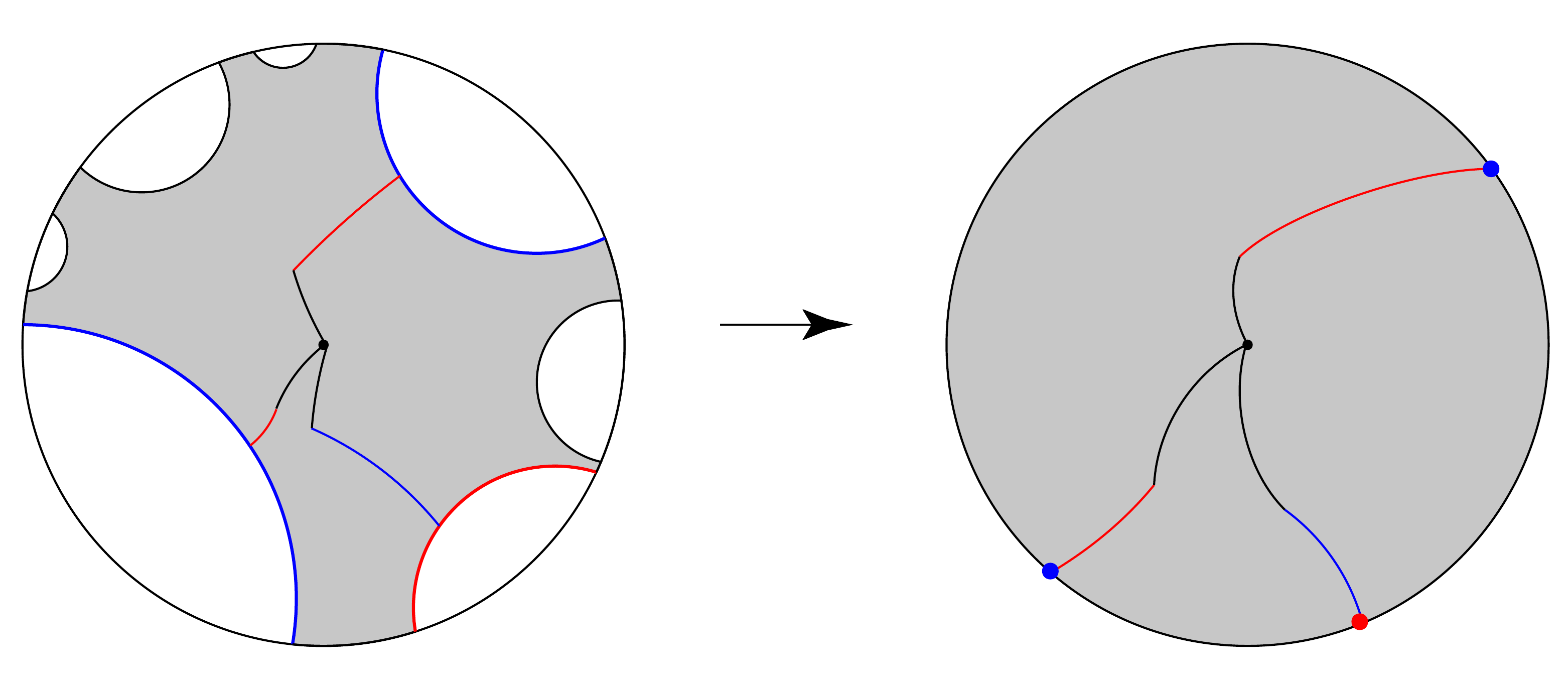}
\caption{}
\label{fig:cyclic-order}
\end{figure}

Moreover, if $\gamma$ is an element of $\pi_1(\Sigma)$, it translates $e$ to another side $\gamma e$ of $p(\lambda)$. We have $g_{\gamma e}=\gamma g_e\gamma^{-1}$, so $Q_\lambda(\gamma e)=\gamma Q_\lambda(e)$. By the minimality of the $\pi_1(\Sigma)$ action on $\partial_\infty\lambda$, the $Q_\lambda$-image of all the sides of $p(\lambda)$ are dense in $\partial_\infty\lambda$. Now there is a unique way to define $Q_\lambda$ on $\partial p(\lambda)$ such that it is continuous restricted to $\partial p(\lambda)$. 
That is, for any point $x\in\partial p(\lambda)$, one can find a sequence of sides $e_n$ converging to $x$ and define $Q_\lambda(x)$ as the limit of $Q_\lambda(e_n)$ by Lemma  \ref{lem:density-of-edges}. The limit exists and is independent of $e_n$ because $Q_\lambda$ is monotone and the image of edges is dense. 

To complete the proof of Lemma \ref{lem:c-continuous-extension}, what is left is to check that $Q_\lambda$ is continuous on $\overline{p(\lambda)}$.

For a side $e$ (which is assumed to be contained in a leaf of $\FF_\OO^s$ for concreteness) of $p(\lambda)$ and a point $x\in e$, we take a rectangular neighborhood of $x$ as the image of an embedding $\rho:(0,1)\times[0,1)\to \overline{p(\lambda)}$ such that
\begin{itemize}
\item $\rho(\frac{1}{2}, 0)=x$;
\item $\rho((0,1)\times \{0\})$ is contained in $e$;
\item for any $s\in[0,1)$, $\rho((0,1)\times\{s\})$ is contained in a leaf of $\FF_\OO^s$;
\item for any $t\in(0,1)$, $\rho(\{t\}\times(0,1))$ is contained in a leaf of $\FF_\OO^u$.
\end{itemize}
Such a neighborhood always exists because $\FF_\OO^s$ is regular on the side of $e$ that contains $p(\lambda)$.

Fix a rectangular neighborhood $RN(x)$ of $x$. By Proposition \ref{prop:local-dynamics}, we know that $Q_\lambda(\rho(\{t\}\times(0,1)))$ is asymptotic to $Q_\lambda(l_e)$, the latter being a quasi-geodesic ray in $\lambda$ since it is a lift of a simple closed curve in $\Sigma$. This shows that $Q_\lambda(RN(x))$ is a wedge-shape region in $\lambda$ with one ideal point $Q_\lambda(e)$ (Figure \ref{fig:boundary}).

\begin{figure}[h!]
\labellist
\pinlabel $p(\lambda)$ at 570 730
\pinlabel $Q_\lambda$ at 1030 520
\pinlabel $\lambda$ at 1600 730
\pinlabel $x$ at 300 440
\pinlabel $Q_\lambda(e)$ at 1300 50
\pinlabel $l_x$ at 700 400
\endlabellist
\includegraphics[scale=.15]{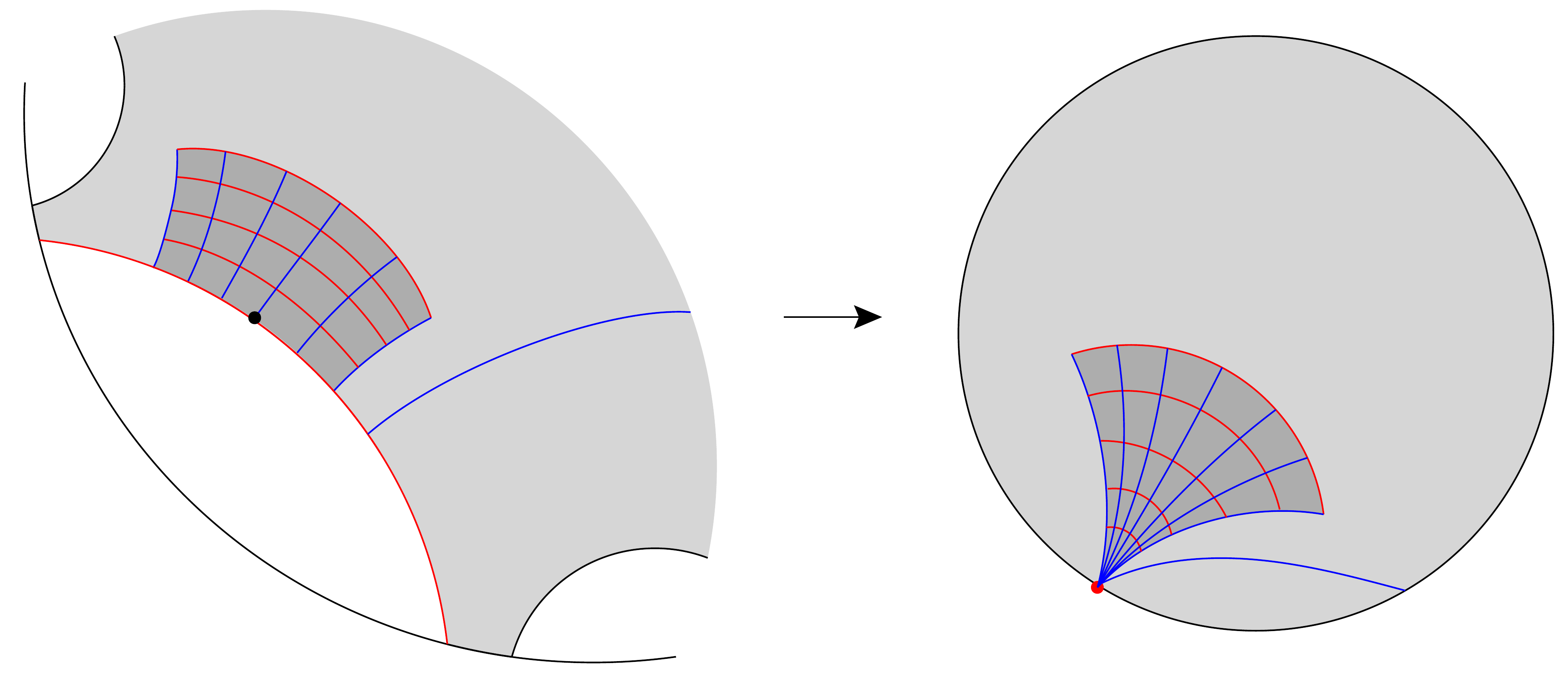}
\caption{The $Q_\lambda$-image of $\FF_\OO^s$-leaves near an unstable boundary leaf $e$}
\label{fig:boundary}
\end{figure}
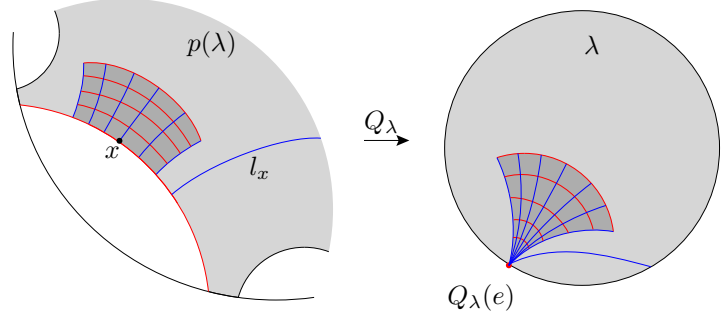

Consider a sequence of points $x_n$ in $p(\lambda)$ converging to $x\in\partial p(\lambda)$. If $x$ is in a side $e$, we can take a rectangular neighborhood of $x$ that will eventually contain $x_n$. The shape of the rectangular neighborhoods under $Q_\lambda$ shows that $Q_\lambda(x_n)$ converge to $Q_\lambda(x)=Q_\lambda(e)$.

If $x\in\partial p(\lambda)\cap\partial\OO$, we can trap $Q_\lambda(x_n)$ using segments of $l_e$, forcing them to converge to the right points. The argument is similar to the one we use to prove that $Q_\lambda$ preserves the cyclic order of the sides. More precisely, take sequences of sides $\{e_m^+\}$ and $\{e_m^-\}$ of $p(\lambda)$ that approximate $x$ from two sides respectively (using Lemma \ref{lem:density-of-edges}). Fixing $m$, take a short segment $l'_{e_m^\pm}\subset l_{e_m^\pm}$ with one endpoint at $x_{e_m^\pm}$ so that $l'_{e_m^+}$ and $l'_{e_m^-}$ are disjoint. Connect the other endpoints of the two segments by a path in $p(\lambda)$, and denote the resulting path by $\alpha_m$. The image $Q_\lambda(\alpha_m)$ is an embedded line in $\wt{\lambda}$ with disjoint endpoints, separating $\wt{\lambda}$ into two half planes. We let the half plane containing $Q_\lambda(x)$ be $H_m(x)$. The fact that $Q_\lambda|_{p(\lambda)}$ is an orientation preserving homeomorphism guarantees that $Q_\lambda(x_n)$ eventually enter $H_m(x)$ for all $m$. We can arrange $H_m(x)$ to be nested so that $\cap_m H_m(x)$ has exactly one ideal point $Q_\lambda(x)$. Since $\{x_n\}$, hence $\{Q_\lambda(x_n)\}$, escape every compact set, we know that $\{Q_\lambda(x_n)\}$ must limit to $Q_\lambda(x)$.

\begin{figure}[!htb]
\labellist
\pinlabel $x$ at 530 20
\pinlabel $Q_\lambda(x)$ at 1370 -10
\endlabellist
\includegraphics[scale=.2]{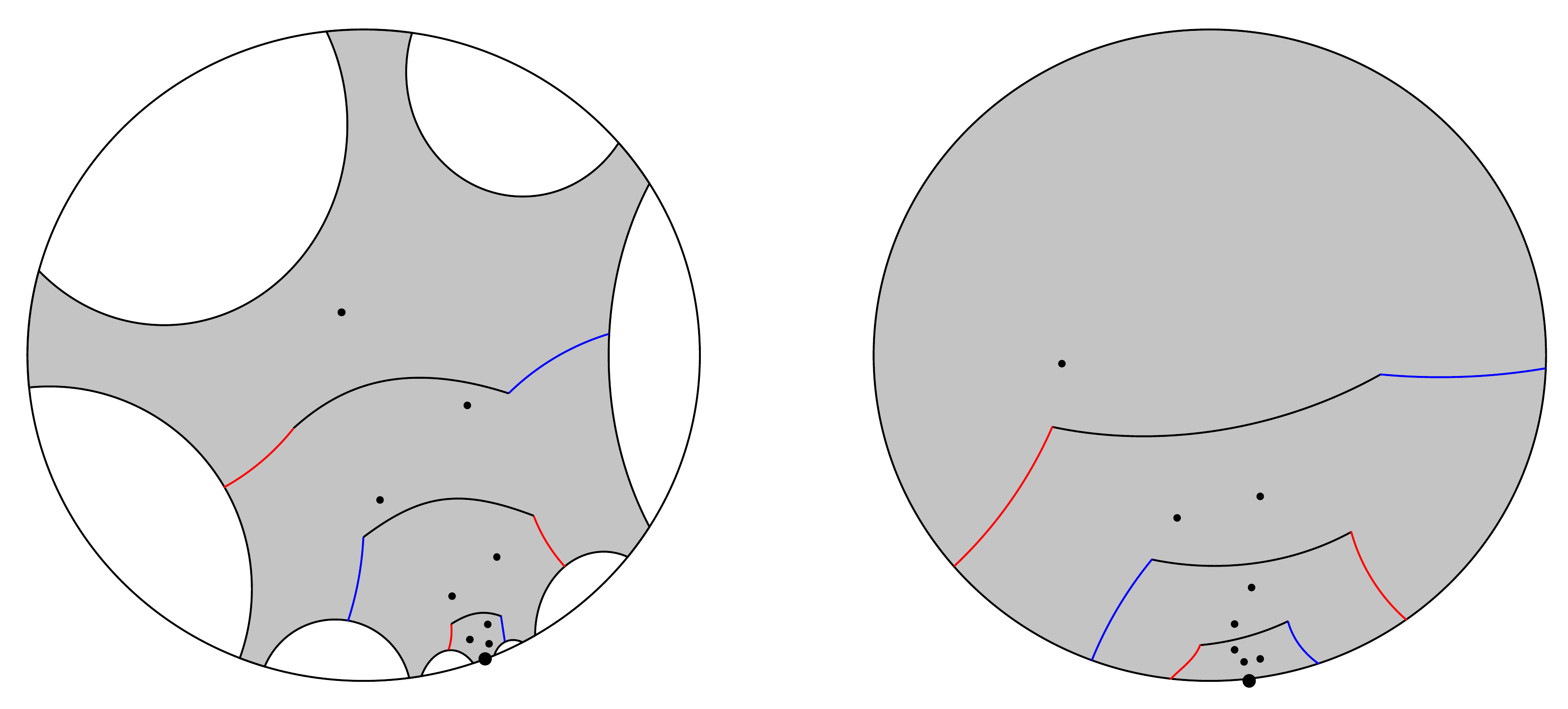}
\caption{}\label{fig:continuity}
\end{figure}

We have proved that $Q_\lambda$ is continuous on $\overline{p(\lambda)}$, finishing the proof of Lemma \ref{lem:c-continuous-extension}.

\end{proof}

\subsection*{Shadows of type-1 leaves}
We now assume $\lambda$ is a type-1 leave contained in a product region $\wt{\Omega}\cong \wt{L}\times\R$ and $\lambda$ covers a non-compact leaf identified with $L$. The first part of the following theorem is a special case of \cite[Proposition 4.1]{Fenley2009}. We gave another proof of this fact in our setting for the sake of completeness. Similar to the case of type-0 leaves, we will use $\partial p(\lambda)$ to denote the boundary of $p(\lambda)$ in $\overline{\OO}$.

\begin{lem}\label{lem:nc-shadow}
    The shadow $p(\lambda)$ is a proper open subset of $\OO$, bounded by periodic regular leaves of $\FF_\OO^s$ or $\FF_\OO^u$, or faces of singular leaves which are regular on the side containing $p(\lambda)$. We call a component of the boundary of $p(\lambda)$ (as a subset of $\OO$) a side of $p(\lambda)$.
        
    Moreover, if a side $e$ of $p(\lambda)$ is contained in a leaf of $\FF_\OO^s$, then there is a type-0 leaf $\mu$ negatively adjacent to $\lambda$ (see Section \ref{subsec:depth-1} for the definition) such that $e$ is also a side of $p(\mu)$. If $e$ is contained in a leaf of $\FF_\OO^u$, then there is a type-0 leaf $\mu$ positively adjacent to $\lambda$ such that $e$ is also a side of $p(\mu)$.
\end{lem}

\begin{proof}
Suppose $z\in\OO$ is a boundary point of $p(\lambda)$. Let $\sigma_z=p^{-1}(z)$ be the $\wt{\phi}$-orbit that projects to $z$. The orbit $\sigma_z$ cannot intersect any product region that is incomparable to $\wt{\Omega}(\lambda)$, which is the product region containing $\lambda$. This is because if $\sigma_z$ intersects such a product region $\wt{\Omega}$, so are the $\wt{\phi}$-orbits close enough to $\sigma_z$, which forces them to be disjoint from $\wt{\Omega}(\lambda)$. Then $\sigma_z$ will not be in the boundary of $p(\lambda)$. Similarly, it cannot intersect any type-0 leaf incomparable to $\wt{\Omega}(\lambda)$ either.

Since $p(\lambda)$ is open, $z$ is not contained in $p(\lambda)$. The orbit $\sigma_z$ induces an oriented path $\gamma_z$ in the dual graph $\Lambda^*$ (Section \ref{sec:preliminaries}). Every vertex in $\gamma_z$ is comparable to $\wt{\Omega}(\lambda)$ by the previous paragraph, but $\wt{\Omega}(\lambda)$ is not in $\gamma_z$. Since $\Lambda^*$ is simply connected, the path $\gamma_z$ cannot be bi-infinite. That means the orbit $\sigma_z$ will eventually stay in one product region in either the positive or the negative direction.

If $\wt{\Omega}(\lambda)$ is in the positive side of $\sigma_z$, then $\sigma_z$ eventually stays in a product region, denoted by $\wt{\Omega}_0$, in the positive direction. Take the unique shortest oriented path $\gamma$ from $\wt{\Omega}_0$ to $\wt{\Omega}(\lambda)$ in $\Lambda^*$ and let $\wt{\Omega}_1$ be the last product region in $\gamma$ before $\wt{\Omega}(\lambda)$. There is a unique type-0 leaf $\mu$ that is positively adjacent to $\wt{\Omega}_1$ and negatively adjacent to $\wt{\Omega}(\lambda)$. Any $\wt{\phi}$-orbit intersecting both $\wt{\Omega}_0$ and $\lambda$ has to enter $\wt{\Omega}(\lambda)$ via $\mu$. Conversely, every $\wt{\phi}$-orbit intersecting $\mu$ also meets $\lambda$. This shows that if an orbit $\sigma$ intersects $\wt{\Omega}_0$, then $p(\sigma)\in p(\mu)$ if and only if $p(\sigma)\in p(\lambda)$. Since $p(\wt{\Omega}_0)$ is open, we have also showed that if $x$ is in $p(\wt{\Omega}_0)$, then $x\in\partial p(\lambda)$ if and only if $x\in\partial p(\mu)$.

Any orbit close enough to $\sigma_z$ will intersect $\wt{\Omega}_0$ as well, so $z$ is in the boundary of $p(\mu)$. Let $e_z$ be the side of $p(\mu)$ containing $z$. Note that all the orbits in the same $\wt{\FF}^s$-leaf as $\sigma_z$ are positively asymptotic to $\sigma_z$. In particular, they intersect $\wt{\Omega}_0$ and are disjoint from $\mu$. We see that $e_z$ is contained in a leaf of $\FF_\OO^s$, and using a similar argument to the last paragraph we have $e_z\subset\partial p(\lambda)$.

To summarize, we have showed the following: for a point $z\in\partial p(\lambda)$, if $\wt{\Omega}(\lambda)$ is in the positive side of $p^{-1}(z)$, then there is type-0 leaf $\mu$ negatively adjacent to $\lambda$ and a side $e_z\in \partial p(\mu)$ such that $z\in e_z$ and $e_z\subset \partial p(\lambda)$. One can apply the same argument to the case where $\wt{\Omega}(\lambda)$ is in the negative side of $p^{-1}(z)$, finishing the proof.

\end{proof}

Similar to the case of type-0 leaves, we define $\overline{p(\lambda)}$ to be $p(\lambda)\cup \partial p(\lambda)$. The following lemma is the type-1 version of Lemma \ref{lem:c-continuous-extension}. 

\begin{lem}\label{lem:nc-continuous-extension}
Let $\lambda$ be a type-1 leaf of $\wt{\FF}$. Then the homeomorphism $p^{-1}|_{p(\lambda)}:p(\lambda)\to\lambda$ extends continuously to a map $Q_\lambda$ from $\overline{p(\lambda)}$ to $\partial_\infty \lambda$. The map restricted to $\partial{p(\lambda)}$ is a monotone map to $\partial_\infty\lambda$ with core $\overline{p(\lambda)}\cap\partial\OO$.
\end{lem}

\begin{proof}

We define $Q_\lambda=(p|_\lambda)^{-1}$ in $p(\lambda)$ and extend $Q_\lambda$ to the boundary following the same strategy as in the proof of Lemma \ref{lem:c-continuous-extension}.

Let $e$ be any side of $p(\lambda)$, and we assume it to be a leaf of $\FF_\OO^s$ for concreteness. By Lemma \ref{lem:nc-shadow} $e$ is also a side of the shadow of a type-0 leaf $\mu$ negatively adjacent to $\lambda$. Take $g_e$, $x_e\in e$ and $l_e$ as in Proposition \ref{prop:local-dynamics}. We orient $l_e$ to be pointing towards $e$. By \cite[Theorem C]{Fenley2009}, each ray of $\wt{\FF}^s\cap\lambda$ or $\wt{\FF}^u\cap\lambda$ has a well-defined endpoint at $\partial_\infty\lambda$. In particular, $Q_\lambda(l_e)$ has a well-defined forward endpoint at infinity.

Let $x'\in e$ be a point different from $x_e$ and set $l'$ to be the intersection of $p(\mu)$ and the ray of $\FF_\OO^u(x')$ containing $x'$. We orient $l'$ to be pointing towards $e$ as well. By Proposition \ref{prop:local-dynamics}, $Q_\mu(l')$ is forward asymptotic to $Q_\mu(l_e)$ under the orientation induced from $l'$ and $l_e$. Since $\FF_\OO^s$ is regular near $e$ on the side of $p(\lambda)$, there is a pair of points $p_e\in l_e$ and $p'\in l'$ so that $p'\in\FF_\OO^s(p_e)$. Let $\sigma_e$ and $\sigma'$ be the $\wt{\phi}$-orbits corresponding to $p_e$ and $p'$ respectively. Since they lie in the same leaf of $\wt{\FF}^s$, $\sigma_e$ and $\sigma'$ are positively asymptotic. It follows that the distance between of $\sigma_e\cap\lambda$ and $\sigma'\cap\lambda$ is bounded by the distance between $\sigma_e\cap\mu$ and $\sigma'\cap\mu$ up to a uniform constant multiple. In other words, $d_\lambda(Q_\lambda(p_e),Q_\lambda(p'))$ is coarsely bounded by $d_\mu(Q_\mu(p_e),Q_\mu(p'))$. Moreover, the pair $\{p_e, p'\}$ can be chosen arbitrarily close to $e$. Hence, $Q_\lambda(l')$ and $Q_\lambda(l_e)$ are forward asymptotic. If $RN(x')$ is a rectangular neighborhood of $x'$ in $p(\mu)$, then we have showed that $Q_\lambda(RN(x'))$ is a wedge-shape domain in $\lambda$ meeting $\partial_\infty\lambda$ at exactly one point. Roughly speaking, this means Figure \ref{fig:boundary} is also a correct picture when $\lambda$ is a type-1 leaf.

We now define $Q_\lambda$ on each side $e$ of $p(\lambda)$ to be the constant map to the forward endpoint of $Q_\lambda(l_e)$. 

\begin{claim}
Different sides are mapped to different point by $Q_\lambda$.
\end{claim}

\begin{proof}[Proof of the claim]

Recall that $\wt{\Omega}(\lambda)$ is the product region containing $\lambda$, and it covers a fibered region $\Omega(\lambda)$ in $M$ with fiber $L$ and the monodromy $h:L\to L$. Since $M$ is atoroidal, $h$ is an atoroidal end-periodic map. The fundamental group of $\Omega(\lambda)$ is isomorphic to the semidirect product $\Z\ltimes\pi_1(L)$ with $\Z$ acting on $\pi_1(L)$ by $h_*$. The group $\pi_1(\Omega(\lambda))$ stabilizes $p(\lambda)$ and we can define an action of $\pi_1(\Omega(\lambda))$ on $\lambda$ by $g(x):=Q_\lambda\circ g\circ p(x)$ for $g\in\pi_1(\Omega(\lambda))$ and $x\in\lambda$. If the element $g$ has a trivial $\Z$-factor, then the action of $g$ is a covering transformation. Otherwise, $g$ acts as a lift of some power of the monodromy $h$. 

Let $e_1$ and $e_2$ be different sides of $p(\lambda)$, and take $g_{e_i}$ and $x_{e_i}$ be as before for $i=1,2$. By Proposition \ref{lem:nc-shadow}, there is a type-0 leaf $\mu_i$ adjacent to $\lambda$ so that $e_i$ is a side of $p(\mu_1)$. The type-0 leaf $\mu_i$ covers a compact leaf $\Sigma_i\subset\partial\Omega(\lambda)$. The element $g_{e_i}$, being a deck transformation for the covering map $\mu_i\to\Sigma_i$, can be viewed as an element of $\pi_1(\Omega(\lambda))$. Therefore, $g_{e_i}$ stabilizes $p(\lambda)$ and acts on $\lambda$ either as a hyperbolic isometry or as a lift of some power of $h$. In both cases, the action extends continuously to an automorphism of $\partial_\infty{\lambda}$ by \cite{CCHomot}. Since $g_{e_i}$ stabilizes $l_{e_i}$ in $p(\lambda)$, we know $g_{e_i}$ stabilizes $Q_\lambda(l_{e_i})$ in $\lambda$, hence also the point $Q_\lambda(e_i)$ by the definition of $Q_\lambda(e_i)$.

Suppose for contradiction that $Q_\lambda(e_1)=Q_\lambda(e_2)=:q$. Then $g_{e_1}$ and $g_{e_2}$ fixes the same point $q$ on $\partial \lambda$. Note that $g_{e_1}$ and $g_{e_2}$ represents different elements in $\pi_1(\Omega(\lambda))$ and not a power of the other by Lemma \ref{lem:stabilizer}. Also note that for any lift $\wt{h}:\lambda\to\lambda$ of some power of the monodromy $h$, the action of $\wt{h}$ on $\partial\lambda$ will not fix any hyperbolic fixed points of $\pi_1(L)$ because $h$ is atoroidal. Therefore, if one of $g_{e_i}$ is a hyperbolic isometry, then $g_{e_1}$ and $g_{e_2}$ have no common fixed point on $\partial_\infty\lambda$. If both of them are a lift of some power of $h$, up to taking powers we may assume that they are lifts of the same power of $h$. Then there is an element $\gamma\in\pi_1(\lambda)$ such that $g_{e_1}=\lambda g_{e_2}$. The common fixed point $q$ will also be fixed by $\lambda$, a contradiction. We have proved that $g_{e_1}$ and $g_{e_2}$ do not have common fixed points, which indicates that $Q_\lambda(e_1)\neq Q_\lambda(e_2)$.
\end{proof}

The map $Q_\lambda$ preserves the cyclic order of the sides by the same reason as in the proof of \ref{lem:c-continuous-extension}, and the image is dense by the minimality of the $\pi_1(L)$-action on $\partial_\infty\lambda$. Indeed, the hyperbolic structure on $L$ has bounded injectivity radius, so the limit set of $\pi_1(L)$ is the entire circle at infinity. 

We also have the following lemma, analogous to Lemma \ref{lem:density-of-edges}.

\begin{lem}\label{lem:nc-density-of-edges}
The intersection $\partial p(\lambda)\cap \partial\OO$ is nowhere dense in $\partial\OO$.
\end{lem}

\begin{proof}
Let $e$ be a side of $p(\lambda)$. By Lemma \ref{lem:nc-shadow}, $e$ is a side of a shadow of a type-0 leaf. So we can take $g_e$ as in Proposition \ref{prop:local-dynamics}. By the discussion above, $g_e$ stabilizes $p(\lambda)$, and the action near $e$ has the desired expanding-contracting dynamics because of Proposition \ref{prop:local-dynamics}. The rest of the proof is the same as that of Lemma \ref{lem:density-of-edges}.
\end{proof}

Now the map $Q_\lambda$ can be extended to $\partial p(\lambda)$ continuously, and the extension is continuous by the nice property of rectangular neighborhoods in type-1 leaves and the same argument as the proof of Lemma \ref{lem:c-continuous-extension}.

\end{proof}

We record a lemma that will be useful in later sections.

\begin{lem}\label{lem:disjoint-edges}
If $\mu$ is a type-0 leaf adjacent to a type-1 leaf $\lambda$, $e_\lambda$ is a side of $p(\lambda)$ and $e_\mu$ is a side of $p(\mu)$, then either $e_\lambda=e_\mu$ or $\partial e_\lambda\cap\partial e_\mu=\varnothing$ holds. Here the boundary $\partial e$ of a side $e$ is the set of endpoints of $e$ on $\partial\OO$.
\end{lem}

\begin{proof}
Suppose $e_\lambda\neq e_\mu$ but they share an endpoint $\xi$ in $\partial\OO$. Then the element $g_{e_\mu}$ stabilizes $\xi$. By the discussion in the proof of the claim above, $g_{e_\mu}$ also stabilizes $p(\lambda)$. It follows that $g_{e_\mu}$ is in the stabilizer of $e_\lambda$, contradicting Lemma \ref{lem:stabilizer}.
\end{proof}

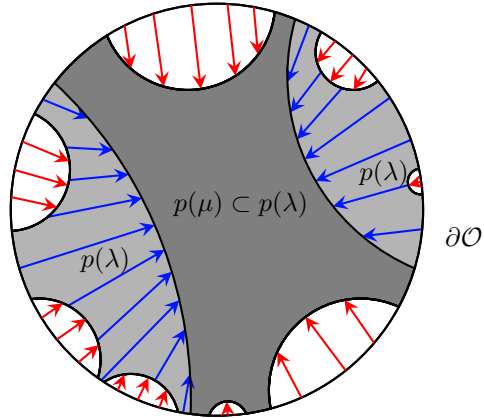
\begin{figure}[htb]
\labellist
\pinlabel $p(\mu)\subset p(\lambda)$ at 320 320
\pinlabel $p(\lambda)$ at 150 250
\pinlabel $p(\lambda)$ at 500 360
\pinlabel $\partial\OO$ at 600 280
\endlabellist
\includegraphics[scale=.3]{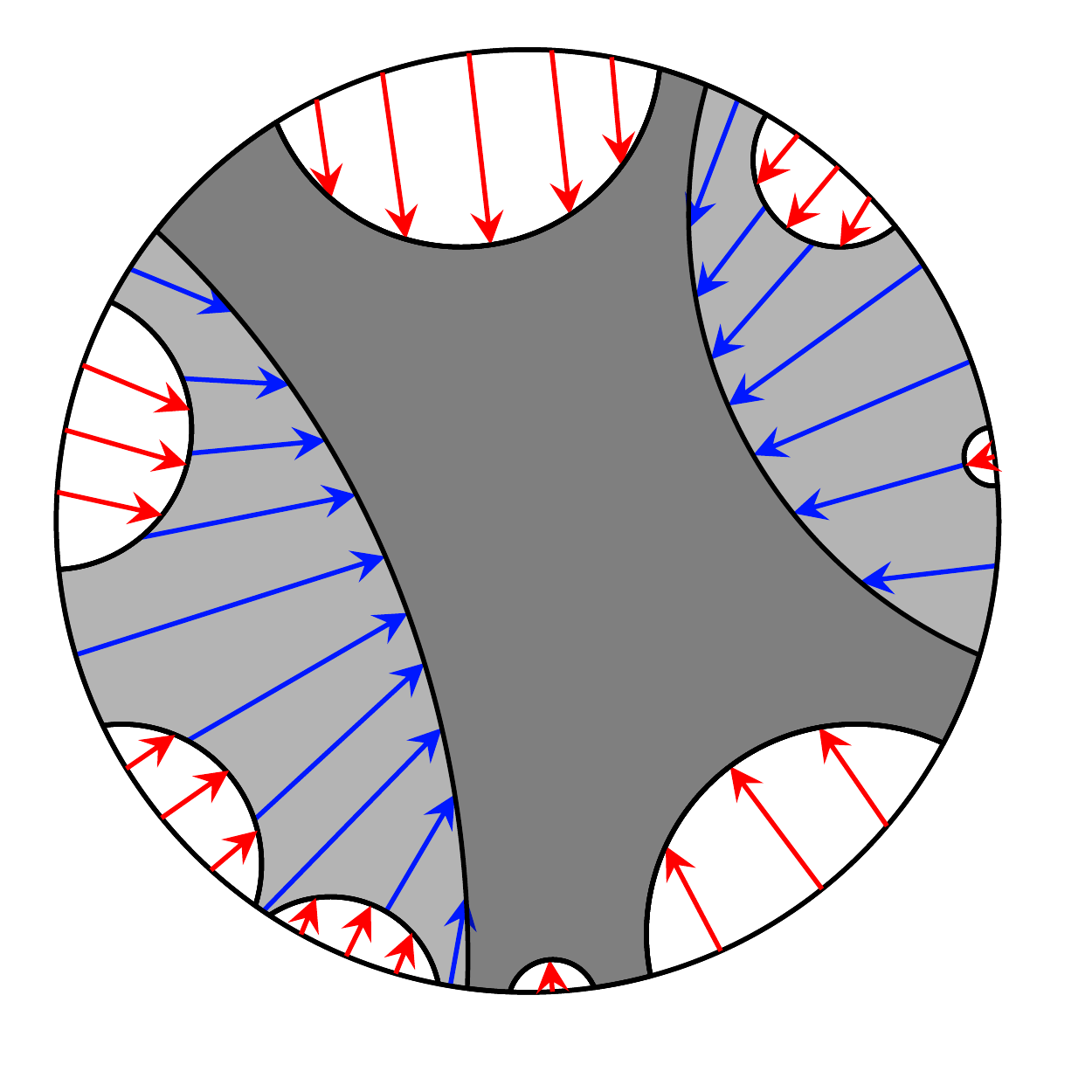}
\caption{The union of shaded area is $p(\lambda)$, and the heavily shaded area is $p(\mu)$, which is a subset of $p(\lambda)$. The red arrows represent the map $Q_\lambda$ and the blue ones represent $Q_{\lambda\mu}$. One should think of each side of $p(\lambda)$ or $p(\mu)$ as a single point at infinity.}
\label{pic:restriction}
\end{figure}

For any leaf $\lambda$ of $\wt{\FF}$, either type-0 or type-1, we define a map $I_\lambda:\partial\OO\to\partial_\infty \lambda$ as follows. For any $\zeta\in\partial\OO$, if $\zeta$ is in $\overline{p(\lambda)}$, set $I_\lambda(\zeta)=Q_\lambda(\zeta)$. If $\zeta$ is contained in an open interval $V_\zeta$ in $\partial\OO\backslash\overline{p(\lambda)}$, then there is a boundary leaf $e$ of $p(\lambda)$ with the same endpoints as $V_\zeta$. In this case, we define $I_\lambda(\zeta)$ to be $Q_\lambda(e)$. It is immediate from the definition that the set of maps $I_\lambda$ are continuous and $\pi_1(M)$-equivariant, i.e. for any $g\in\pi_1(M)$, $\lambda\in\wt{\FF}$ and $x\in\partial\OO$, we have $gI_\lambda(x)=I_{g\lambda}(gx)$, where the action $g: \partial_\infty\lambda\to\partial_\infty(g\lambda)$ is induced by the isometry $g:\lambda\to g\lambda$.

If $\lambda$ is a type-1 leaf and $\mu$ is a type-0 leaf adjacent to $\lambda$, there is a continuous surjection $I_{\lambda\mu}:\partial_\infty\lambda\to\partial_\infty\mu$ such that $I_\mu=I_{\lambda\mu}\circ I_\lambda$. More precisely, for any point $\xi\in\partial_\infty\lambda$, $I_{\lambda\mu}(\xi)$ is defined to be $Q_\mu\circ Q_\lambda^{-1}(\xi)$. The maps $I_\lambda$ and $I_{\lambda\mu}$ have clear geometric meaning as shown in the schematic Figure \ref{pic:restriction}.

\section{Markers and universal circles}\label{sec:universal-circle}

The outline of this section is the following. We first recall the definition of a universal circle (Definition \ref{def:universal-circle}) and prove Theorem \ref{thm:ideal-boundary}. Then we review a construction from \cite{Calegari:2003aa} of a particular universal circle $\Su_\mathrm{left}$, which we call the universal circle from leftmost sections, for any taut foliation. The construction will be then examined carefully for our depth-one foliation $\FF$. The circle $\Su_\mathrm{left}$ arise from a collection of special sections, called the leftmost sections, of a circle bundle $E_\infty$ over $\Lambda$ whose fibers are the circle at infinity of the leaves. The punchlines of this section are Lemma \ref{lem:0-to-1} and Lemma \ref{lem:U-gaps}, where we show that the leftmost sections can be determined by the structure of $\partial\OO$ developed in Section \ref{sec:shadows}.

The axiomatic definition of a universal circle for $\FF$ is due to Calegari-Dunfield \cite{Calegari:2003aa}, and it is stated as follows.

\begin{defn}[Universal circle]\label{def:universal-circle}
A \textbf{universal circle} for $\FF$ is a circle $\Su$ with a $\pi_1(M)$-action and a monotone map $U_\lambda:\Su\to\partial_\infty\lambda$, called a \textbf{structure map} for any leaf $\lambda$ of $\wt{\FF}$ such that:
\begin{itemize}
\item [(1)] for any leaf $\lambda$ and any $\gamma\in\pi_1(M)$, the following diagram commutes:
\[
\xymatrix{
\Su\ar[d]_{\gamma}\ar[r]^{U_\lambda} & \partial_\infty(\lambda)\ar[d]^{\gamma}\\
\Su\ar[r]^{U_{\gamma\lambda}} & \partial_\infty(\gamma\lambda)\\
}
\]
\item [(2)] if $\lambda$ and $\mu$ are incomparable leaves, then the core of $U_{\lambda_1}$ is contained in a single gap of $U_{\lambda_2}$, and vice versa.
\end{itemize}

Two universal circles $\{\Su,U_\lambda\}$ and $\{\Su',U'_\lambda\}$ are isomorphic if there is a $\pi_1(M)$-equivariant homeomorphism $h:\Su\to\Su'$ such that $U'_\lambda\circ h=U_\lambda$ for all $\lambda$.
\end{defn}

\begin{proof}[Proof of Theorem \ref{thm:ideal-boundary}]
All the conditions of a universal circle in Definition \ref{def:universal-circle} are obvious by the way we define $I_\lambda$ except for condition (2). Suppose $\lambda$ and $\mu$ are incomparable, then their shadows are disjoint. Otherwise, there is an orbit of the flow $\wt{\phi}$ intersecting both leaves, contradicting their incomparability. Condition (2) is easily seen to be satisfied.
\end{proof}

In \cite{Calegari:2003aa}, Calegari-Dunfield describe an explicit construction of a universal circle $\Su_{\mathrm{left}}$ for any taut foliation. We briefly review their construction below. For simplicity, we will stick to our $\FF$ instead of more general taut foliations.

The bundle $E_\infty$ is a circle bundle over $\Lambda$ whose fiber at any leaf $\lambda$ is the circle at infinity $\partial_\infty\lambda$. The topology of $E_\infty$ is defined as follows. For any transversal $\tau$ of $\wt{\FF}$, $\tau$ embeds into $\Lambda$ and we identify $\tau$ with its embedding image in $\Lambda$. The unit tangent bundle of $\wt\FF$ restricted to $\tau$ is the circle bundle $UT\wt{\FF}|_\tau$, and there is a natural map $UT\wt{\FF}|_\tau\to E_\infty|_\tau$ sending a tangent vector of a leaf to the ideal point it points towards. We require the map to be a homeomorphism. It is shown in \cite{Calegari:2003aa} that this topology is well-defined, i.e. independent of the choice of $\tau$.

Since $\FF$ is a taut foliation, there is an $\epsilon_1>0$ such that every leaf of $\wt{\FF}$ is quasi-isometrically embedded in its $\epsilon_1$-neighborhood by \cite[Lemma 2.4]{Calegari:2003aa}. By the structure of depth-one foliations there is a constant $\epsilon_2>0$ so that the $\epsilon_2$-neighborhood of $\FF^0$ is contained in a spiraling neighborhood. Fix $\epsilon_0$ to be $\min\{\epsilon_1/3,\epsilon_2\}$.

\begin{defn}
A \textbf{marker} for $\wt{\FF}$ is an embedding 
\[
m:[0,1]\times\R_{\geq0}\to\wt{M}
\]
such that
\begin{itemize}
\item for any $x\in[0,1]$, $m(\{x\}\times\R_{\geq0})$ is a geodesic ray in a leaf of $\wt{\FF}$;
\item for any $y\in\R_+$, $m([0,1]\times\{y\})$ is a transversal with length bounded by $\epsilon_0$.
\end{itemize}
Any marker $m$ gives a section $s$ of $E_\infty|_\tau$, where $\tau$ is the image of $m([0,1]\times\{y\})$ in $\Lambda$, such that for any leaf $\lambda\in\tau$, $s(\lambda)$ is the ideal endpoint of $\mathrm{Image}(m)\cap\lambda$. The image of $\tau$ under $s$ is called the \textbf{end} of $m$.
\end{defn}

Note that our choice of $\epsilon_0$ is different from but no larger than the constant chosen in \cite{Calegari:2003aa}. Shrinking the constant will not affect the constructions and the main results in their paper.

\begin{lem}[\cite{Calegari:2003aa}]\label{lem:marker-disjointness}
Given two marker ends, either they are disjoint or their union is an embedded closed interval in $E_\infty$ transverse to fibers.
\end{lem}

A point $\xi\in\partial\lambda$ is called a \textbf{marker endpoint} if there is a marker $m$ so that the end of $m$ intersect $\partial_\infty\lambda$ at $\xi$. The following theorem is originally announced by Thurston in an unfinished manuscript \cite{ThuCircle2}, and the proof is carefully written down in \cite{Calegari:2003aa}. Heuristically, it says that the leaves of $\wt{\FF}$ stay close in many directions.

\begin{thm}[Thurston's Leaf Pocket Theorem]\label{thm:leaf-pocket}
For any leaf $\lambda$ of $\wt{\FF}$, the set of marker endpoints in $\partial_\infty\lambda$ is dense.
\end{thm}

For any leaf $\lambda$ of $\wt{\FF}$ and any point $\xi\in\partial_\infty\lambda$, there is a special section $s_\xi$ of $E_\infty$, called the \textbf{leftmost section starting from} $\xi$, built as follows.

In $\Lambda$, there is a neighborhood $\tau$ of $\lambda$ homeomorphic to a closed interval, and $E_\infty|\tau$ is a cylinder. We adopt the convention that the flow $\wt{\phi}$ is flowing upwards, and we are facing the cylinder $E_\infty|\tau$ from the outside. Take a finite collection $C$ of marker ends in $E_\infty|\tau$ so that each fiber intersects at least one. This is possible by Theorem \ref{thm:leaf-pocket}. We build a path $\alpha_C$ by starting from $\xi$, heading left horizontally in a fiber until we hit the first marker end in $C$ and following the marker end to move upwards. After we reach the top of the marker end, we turn left again staying in a fiber until we hit the next marker end in $C$, and follow the same rules to move on until we reach the top of $E_\infty|_\tau$. We call this the leftmost-up rule, following \cite{Calegari:2003aa}. We can also move downwards from $\xi$, but in the rightmost-down way. This procedure gives us a staircase path $\alpha_C$ in $E_\infty|_\tau$, which is an approximation to $s_\xi$ (Figure \ref{fig:staircase}).

\begin{figure}[h!]
\labellist
\pinlabel $\xi$ at 150 25
\endlabellist
\includegraphics[scale=.8]{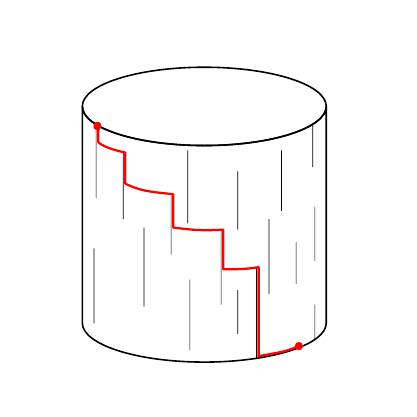}
\caption{Approximating leftmost section on $E_\infty|_\tau$ by starting from $\xi$ and go leftmost-up.}\label{fig:staircase}
\end{figure}

To go from the staircase approximations to the leftmost section $s_\xi$, we define $s_\xi|_\tau$ to be the (rightmost) supremum above $\xi$ and the (leftmost) infimum below $\xi$ among all possible $\alpha_C$. To be precise, we view $E_\infty|_\tau$ as $\tau\times (\R/\Z)$. For a leaf $\mu\in\tau$ above $\lambda$, we define 
\[
s_\xi(\mu)=\sup_C(\min (\alpha_C\cap\partial_\infty\mu)).
\]
For $\mu\in\tau$ below $\lambda$, define 
\[
s_\xi(\mu)=\inf_C(\max (\alpha_C\cap\partial_\infty\mu)).
\]
It was proved in \cite{Calegari:2003aa} that the supremum and the infimum exist, and $s_\xi$ is indeed a continuous section of $E_\infty$ over $\tau$. We can define $s_\xi$ for all leaves comparable to $\lambda$ following this procedure.

Finally, we can branch out in $\Lambda$ by turning around to reach incomparable leaves where $s_\xi$ is not yet defined. More precisely, suppose $\mu$ is a leaf incomparable to $\lambda$. There is a sequence of leaves
\[
\lambda_0=\lambda,\lambda_1,\cdots,\lambda_n=\mu
\]
so that $\lambda_{2i}$ and $\lambda_{2i+1}$ are comparable, and $\lambda_{2i+1}$ and $\lambda_{2i+2}$ are non-separated. To illustrate the idea, we assume $\lambda_1$ is above $\lambda$ (see Figure \ref{fig:path} for the case when $n=9$). In this case, there is a product region $\wt{\Omega}_1$ so that $\lambda_{1},\lambda_2\in\partial^-\wt{\Omega}_1$. Let $l_0$ be the segment $[\lambda,\lambda_1]$ in $\Lambda$, and let $l_1$ be the image of $\wt{\Omega}_1$ in $\Lambda$. By the above construction, $s_\xi$ is already defined over $l_0$, $\lambda_1$ and $l_1$. It is a consequence of Lemma \ref{lem:marker-disjointness} and Theorem \ref{thm:leaf-pocket} that $s_\xi|_{l_1}$ has a well-defined endpoint at $\partial_\infty\lambda_2$ (see Lemma 6.18 of \cite{Calegari:2003aa}). We extend $s_\xi$ to $\lambda_2$ continuously, and follow the rightmost-down rule to define it over the segment $l_2:=[\lambda_2,\lambda_3]$. We continue along the sequence $\lambda_i$ until we have defined $s_\xi(\mu)$. Since the dual graph $\Lambda^*$ is a tree, there is a unique way to reach any incomparable $\mu$ from $\lambda$ through such a sequence of $\lambda_n$. In the end we have a section $s_\xi$ that is well-defined on the whole $\Lambda$. 

The process of extending $s_\xi$ is a process of branching out from $\lambda$ and sweeping $\Lambda$. The values at leaves that are closer to $\lambda$ are defined first, and the values at leaves farther away from $\lambda$ are determined by the closer values. At each point of $\Lambda$, there is a direction of extension of $s_\xi$ that points towards the direction away from $\lambda$, along which $s_\xi$ is defined.

\begin{figure}[h!]
\labellist
\pinlabel $\lambda$ at 20 30
\pinlabel $\lambda_1$ at 5 135
\pinlabel $l_1$ at 19 178
\pinlabel $l_0$ at 5 85
\pinlabel $\lambda_2$ at 50 135
\pinlabel $\l_2$ at 50 105
\pinlabel $\lambda_3$ at 65 80
\pinlabel $\mu$ at 220 190
\endlabellist
\includegraphics[scale=.8]{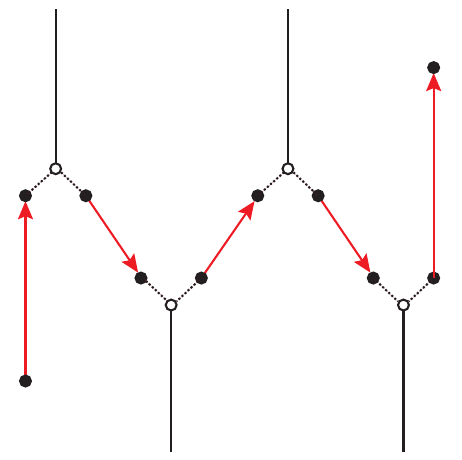}
\caption{}\label{fig:path}
\end{figure}

There is a unique leftmost section starting from any point in $E_\infty$. The set of leftmost sections is denoted by $LS$. The image of two different leftmost sections might coalesce but can never cross each other. If $\ell$ is a line in $\Lambda$, the bundle $E_\infty|_\ell$ is homeomorphic to a cylinder, and the leftmost sections restricted to $\ell$ give embedded lines on $E_\infty|_\ell$ transverse to the fiber. For any three different leftmost sections, there is an embedded line in $\Lambda$ so that the restrictions of the sections to this line have a well-defined cyclic order, and the cyclic order is independent of the choice of the line \cite[Lemma 6.25]{Calegari:2003aa}. The completion of $LS$ with respect to the cyclic order is homeomorphic to a circle, denoted by $\Su_\mathrm{left}$. The fundamental group $\pi_1(M)$ acts naturally on $LS$, and the action extends to an action on $\Su_\mathrm{left}$. For any $\lambda\in\Lambda$, there is a valuation map $U_\lambda:LS\to\partial_\infty\lambda$ given by
\[
U_\lambda(s)=s(\lambda).
\]
The map $U_\lambda$ can be extended to a monotone map $U_\lambda:\Su_\mathrm{left}\to\partial_\infty\lambda$.

\begin{thm}[\cite{Calegari:2003aa}]\label{thm:leftmost-universal-circle}
The circle $\Su_\mathrm{left}$ together with the $\pi_1(M)$-action and the set of structure maps $\{U_\lambda\}_{\lambda\in\Lambda}$ is a universal circle for $\FF$.
\end{thm}

In order to prove Theorem \ref{thm:main}, we need to analyze more carefully what marker ends and the leftmost sections look like on $E_\infty$. We will continue to use the terminologies from Section \ref{sec:shadows}.

We first consider the ends of markers that are contained in a product region. The identification of $\widetilde{\Omega}$ with $\widetilde{L}\times\R$ gives a canonical trivialization of $E_\infty\vert_{\wt{\Omega}}$ as $\partial_\infty\widetilde{L}\times\R$. Here we implicitly use that any homeomorphism between two infinite type surfaces with standard hyperbolic structures extends continuously to a homeomorphism between their boundaries at infinity \cite{CCHomot}. Denote the leaf $\widetilde{L}\times\{t\}$ by $\lambda_t$. Again, each $\lambda_t$ is identified with $\wt{L}$.

\begin{lem}\label{lem:product-markers}
    For any $\xi\in\partial_\infty\lambda_t$, there is an $\epsilon>0$ so that $\{\xi\}\times[t,t+\epsilon]$ is the end of a marker.
\end{lem}

\begin{proof}
Take any point $x\in\lambda_t$ and consider the geodesic ray $\gamma$ from $x$ to $\xi$. Since depth-one leaves in the same fibered region have asymptotic ends, there is an $\epsilon$ such that any flow line of $\wt{\phi}$ between $\lambda_{t}$ and $\lambda_{t+\epsilon}$ has length $<\epsilon_0$. By our choice of $\epsilon_0$, the flow image of $\gamma$ in $\lambda_{s}$ for $s\in(t,t+\epsilon]$ is a family of quasi-geodesics with endpoint $(\xi,s)$, and the quasi-geodesic constants converge to $(1,0)$ as $s$ approaches $t$. We can then pull tight every quasi-geodesic to be a geodesic and take $\epsilon$ even smaller to obtain a genuine marker with the end $\{\xi\}\times[t,t+\epsilon]$. The continuity of the tightened marker is ensured by the continuity of the hyperbolic metric, and the finite-width of the marker is guaranteed by the nice variation of quasi-geodesic constants.
\end{proof}

The next corollary follows immediately from the construction of leftmost sections on comparable leaves.

\begin{cor}[Leftmost sections on a product region are vertical]\label{cor:vertical-section}
    Suppose $s$ is any leftmost section of $E_\infty$. If $s(\lambda_{t_0})=(\xi,t_0)$ for some $t_0\in\R$, then $s(\lambda_{t})=(\xi,t)$ for all $t$.
\end{cor}

\begin{proof}
By the construction of leftmost sections, one can see that any leftmost section $s$ has the following property: any marker end is either contained in or disjoint from $s(\Lambda)$. Since through any point in $E_\infty|_{\wt{\Omega}}$ there is a vertical marker end, the leftmost section $s$ restricted to $\wt{\Omega}$ is forced to be vertical.
\end{proof}

We now consider the ends of markers intersecting a type-0 leaf. Let $\mu$ be a type-0 leaf covering a compact leaf $\Sigma$, and suppose $\mu$ is in the positive boundary of $\widetilde{\Omega}$. The case where $\mu$ is in the negative boundary of $\wt{\Omega}$ is similar.

As in the previous discussion, we identify every depth-one leaf in $\Omega$ with $L$, and every type-1 leaf in $\wt{\Omega}$ with $\wt{L}$. This gives us a homeomorphism between $E_\infty|_{\wt{\Omega}}$ and $\partial_\infty{\wt{L}}\times\R$. Recall that in Section \ref{sec:shadows}, we define a continuous map
\[
I_{\lambda_t\mu}:\partial_\infty\lambda_t\to\partial_\infty\mu
\]
for any $\lambda_t\in\wt{\Omega}$. Under the homeomorphism $\partial_\infty\lambda_t\cong\partial_\infty\wt{L}$, the map $I_{\lambda_t\mu}$ is the same map for any $t$ when viewed as a map from $\partial_\infty\wt{L}$ to $\partial_\infty\mu$. We denote this map by $I_{\wt{\Omega}\mu}$.

In $\Lambda$, $\wt{\Omega}\cup\mu$ is a half open interval. Let $\MM_{\mu,\wt{\Omega}}$ be the set of markers $m$ in $\mu\cup\wt{\Omega}$ with one side lying in $\mu$. By Corollary \ref{cor:vertical-section} and Lemma \ref{lem:marker-disjointness}, the end of such an $m$ intersects $\partial_\infty\mu$ at a single point $\xi_\mu(m)$ and intersects $E_\infty|_{\wt{\Omega}}$ in a vertical segment 
\[
\{(\xi_{\wt{\Omega}}(m),t)|~t>T\}
\]
where $T$ is a constant depending on $m$, and $\xi_{\wt\Omega}(m)$ is a point in $\partial_\infty\wt{L}$ depending on $m$ and independent of $t$ (Figure \ref{fig:special-markers}). 

\begin{figure}[h!]
\labellist
\pinlabel $\mu$ at 900 800
\pinlabel $\wt{\Omega}$ at 920 400
\pinlabel $\partial_\infty\mu$ at -60 800
\pinlabel $\partial_\infty\wt{L}\times\R$ at -100 400
\pinlabel $\xi_\mu(m)$ at 330 770
\pinlabel $\xi_{\wt{\Omega}}(m)\times[T,+\infty)$ at 550 470
\endlabellist
\includegraphics[scale=.2]{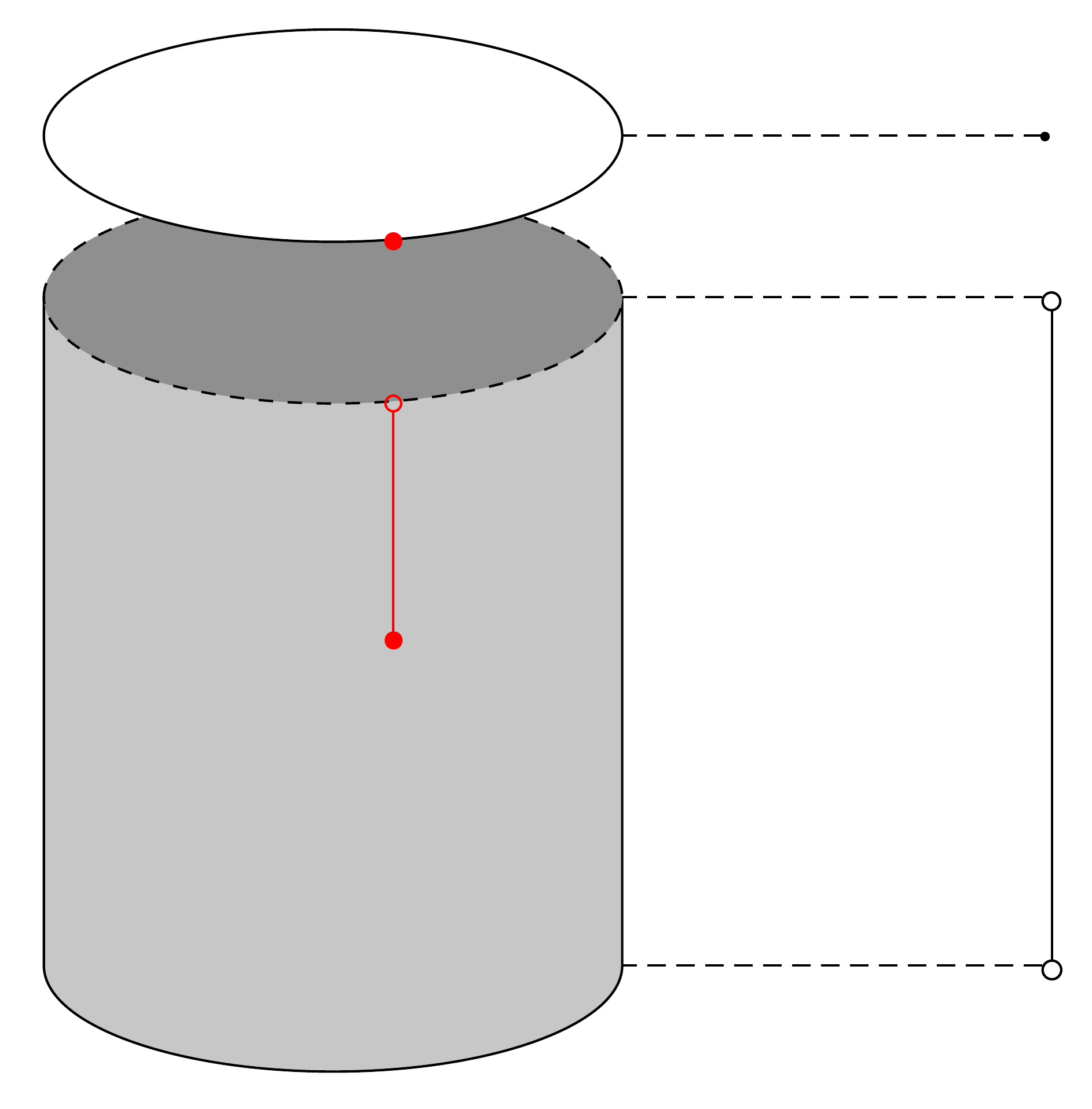}
\caption{The end of a marker $m$ (in red) in $\MM_{\mu,\wt{\Omega}}$.}\label{fig:special-markers}
\end{figure}

Define the set $\xi_{\wt\Omega}(\MM_{\mu,\wt{\Omega}}):=\{\xi_{\wt\Omega}(m)|~m\in\MM_{\mu,\wt{\Omega}}\}\subset\partial_\infty\wt{L}$. Intuitively, these are the directions on $\wt{L}$ in which $\mu$ does not diverge away from $\lambda_t$.

Recall that the positive escaping set $\UU^+\subset L$ is the set of points whose $\phi$-orbit escapes $\Omega$ in positive time. Let $\wt{\UU}^+$ be the preimage of $\UU^+$ in $\wt{L}$. There is a component $\wt{\UU}^+_\mu$ of $\wt{\UU}^+$ so that $x\in\lambda_t$ is a point of $\wt{\UU}^+_\mu$ if and only if $\wt{\phi}(x)$ hits $\mu$. The hitting map $H_\mu:\wt{\UU}^+_\mu\to\mu$ defined by
\[
H_\mu(x)=\wt{\phi}(x)\cap\mu
\]
is a homeomorphism. It is tautological that
\begin{equation}\label{eq:tautological}
p\circ H_\mu=p.
\end{equation}
where $p$ is the projection from $\mu$ to $p(\mu)$ as in Section \ref{sec:shadows}.

\begin{lem}\label{lem:trivial-holonomy}
Let $\alpha^*$ be an oriented simple closed geodesic in $\Sigma$ with trivial $\FF$-holonomy on the side of $\Omega$, and let $\wt{\alpha}^*$ be a lift of $\alpha^*$ to $\mu$. Denote the forward endpoint at infinity of $\wt{\alpha}^*$ by $\partial^+\wt{\alpha}^*$. Then there is a marker $m\in\MM_{\mu,\wt{\Omega}}$ so that $\xi_\mu(m)=\partial^+\wt{\alpha}^*$ and $\xi_{\wt\Omega}(m)\in I_{\wt{\Omega}\mu}^{-1}(\partial^+\wt{\alpha}^*)$.
\end{lem}

\begin{proof}
Perform a homotopy of $\alpha^*$ along $\phi$ into $\Omega$ for a short distance, so that the final image is a simple closed curve on a depth-one leaf in $\Omega$. This is possible because $\alpha^*$ has trivial $\FF$-holonomy on the side of $\Omega$. The full homotopy image is an annulus, denoted by $A$. We lift $A$ to $\wt{M}$ to get a two-ended infinite band $\wt{A}$ in $\wt{\Omega}\cup\mu$ with one side being $\wt\alpha^*$ and the other side on $\lambda_T$ for some $T$. Note that $A$ has bounded width because it is a lift of an annulus.

Fix a base point $x\in\wt{\alpha}^*$ and let $\wt{\alpha}^*(x)$ be the oriented ray in $\wt\alpha^*$ starting from $x$ toward $\partial^+\wt{\alpha}^*$. We restrict the lifted homotopy to the ray $\wt{\alpha}^*(x)$, and the restricted homotopy image is a one-ended infinite band $B\subset \wt{A}$ (Figure \ref{fig:trivial-holonomy}). The band $B$ has finite width between $\mu$ and $\lambda_T$, and we can make the width arbitrarily small by cutting the homotopy at $\lambda_{T'}$ for large enough $T'$. When $t>T'$, the ray $B\cap\lambda_t$ has a well-defined endpoint at infinity because it is contained in a lift of a simple closed curve on a hyperbolic depth-one leaf. Similar to the proof of Lemma \ref{lem:product-markers}, by taking even larger $T'$ and pulling tight the intersection of $B$ and $\lambda_t$ we get a marker $m\in\MM_{\mu,\wt{\Omega}}$. It is clear from the construction that $\xi_\mu(m)=\partial^+\wt{\alpha}$. The point $\xi_{\wt{\Omega}}(m)$ is the endpoint of the ray
\[
B\cap\lambda_t=H_\mu^{-1}(\wt{\alpha}(x)).
\]
By (\ref{eq:tautological}), the projections of $\wt{\alpha}(x)$ and $H_\mu^{-1}(\wt{\alpha}(x))$ to $\OO$ are identical, and both rays escape to infinity in the leaf. This implies $\partial^+\wt{\alpha}\in\partial p(\mu)\cap\partial p(\lambda_t)$ and 
\[
\xi_{\wt{\Omega}}(m)\in I_{\wt{\Omega}\mu}^{-1}(\partial^+\wt{\alpha}).
\]
\end{proof}

\begin{figure}[h!]
\labellist
\pinlabel $\Sigma$ at 200 370
\pinlabel $\alpha^*$ at 420 330
\pinlabel $\wt{\alpha}^*$ at 1320 380
\pinlabel $B$ at 1070 215
\pinlabel $\mu$ at 1500 300
\pinlabel $x$ at 1200 310
\endlabellist
\includegraphics[scale=.26]{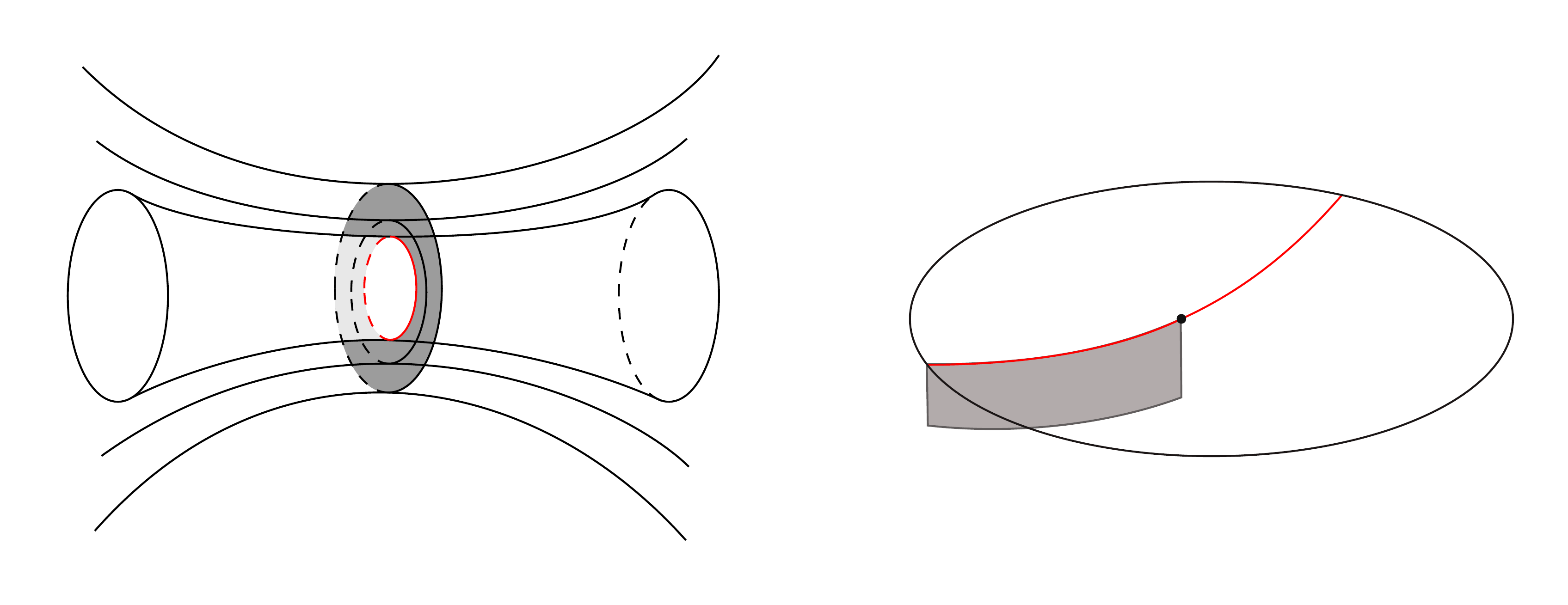}
\caption{}\label{fig:trivial-holonomy}
\end{figure}

To state the next lemma, we need one more definition. A simple closed curve $\beta\subset L$ is called a \textbf{$\Sigma$-juncture} if $\beta$ is a connected positive $f$-juncture, and $\beta$ covers a simple closed curve in $\Sigma$ under the covering map $\UU^+\to\partial^+\Omega$ (note that $\Sigma$ is only a component of $\partial^+\Omega$).

\begin{cor}\label{cor:juncture-endpoint}
Let $\beta\subset L$ be a $\Sigma$-juncture and let $\wt{\beta}$ be a lift of $\beta$ in $\wt{L}$ that lies in $\wt{\UU}^+_\mu$. Then both endpoints of $\wt{\beta}$ are in $\xi_{\wt{\Omega}}(\MM_{\mu,\wt{\Omega}})$.
\end{cor}

\begin{proof}
Suppose $\beta$ covers a simple closed curve $\alpha$ on $\Sigma$. Fix an orientation of $\beta$, which induces an orientation of $\alpha$. The curve $\alpha$ has trivial $\FF$-holonomy on the side of $\Omega$ because $\beta$ is an $f$-juncture. For any $T$, we can view $\wt{\beta}$ as an embedded line in $\lambda_T$. There is a lift $\wt{\alpha}\subset\mu$ of $\alpha$ given by $\wt{\alpha}=\wt{\phi}_{\geq0}(\wt{\beta})\cap\mu$. Let $\partial^+\wt{\alpha}$ be the forward endpoint of $\wt{\alpha}$. Applying Lemma \ref{lem:trivial-holonomy} to the geodesic tightening $\alpha^*$ of $\alpha$, we obtain a marker $m\in\MM_{\mu,\wt{\Omega}}$. But since $\wt{\alpha}$ is obtained by flowing $\wt{\beta}$ to $\mu$, the proof of Lemma \ref{lem:trivial-holonomy} implies that $\xi_{\wt{\Omega}}(m)$ is exactly the forward endpoint of $\wt{\beta}$. The same proof applies for the backward endpoint of $\wt{\beta}$, proving the lemma.

\end{proof}

\begin{lem}\label{lem:marker-closure}
    The closure in $\partial_\infty\wt{L}$ of $\xi_{\wt{\Omega}}(\MM_{\mu,\wt{\Omega}})$ is the limit set $\mathcal{L}(\wt{\UU}^+_\mu)$ of $\wt{\UU}^+_\mu$.
\end{lem}

\begin{proof}
Suppose $m$ is a marker in $\MM_{\mu,\wt{\Omega}}$. The intersection of the image of $m$ with any $\lambda_t$ is a ray $r_t$. After identifying $\lambda_t$ with $L$, we claim that $r_t$ in contained in $\wt{\UU}^+_\mu$. This is because the image of $m$ is $\epsilon_0$-close to $\mu$. By our choice of $\epsilon_0$, the $\epsilon_0$-neighborhood of $\Sigma$ in $M$ is contained in an $\FF$-spiraling neighborhood $N(\Sigma)$ of $\Sigma$, and every $\phi$-orbit intersecting the spiraling neighborhood will hit $\Sigma$. So $r_t$ is contained in $\wt{\UU}^+_\mu$ and the ideal endpoint of $r_t$, which is exactly $\xi_{\wt{\Omega}}(m)$ by definition, will then be in $\mathcal{L}(\wt{\UU}^+_\mu)$. This shows $\xi_{\wt{\Omega}}(\MM_{\mu,\wt{\Omega}})\subset \mathcal{L}(\wt{\UU}^+_\mu)$.

By Corollary \ref{cor:juncture-endpoint}, it now suffices to show that the endpoints of lifts of any $\Sigma$-juncture in $\wt{\UU}^+_\mu$ are dense inside $\mathcal{L}(\wt{\UU}^+_\mu)$. We first recall some known facts and constructions.  By \cite{CANTWELL_CONLON_FENLEY_2021}, the geodesic tightenings of a system of positive $f$-junctures limit to the negative Handel-Miller geodesic lamination $\Lambda_{\text{HM}}^-$ under negative iterations of $f$, and $\Lambda_{\text{HM}}^-$ is independent of the choice of junctures. On the other hand, the intersection of $\FF^s$ with $L$ induces a singular foliation $\FF_L^s$ on $L$. We define $W^-$ as the restriction of $\FF_L^s$ to the complement of $\UU^+$. The complement of $\UU^+$ is saturated by leaves of $\FF_L^s$ and $W^-$ is a singular sublamination of $\FF_L^s$ \cite{LMTEndpA}. Let $\wt{W}^-$ be the lift of $W^-$ to $\wt{L}$. The singular lamination $\wt{W}^-$ determine an abstract lamination on $\partial_\infty\wt{L}$ by a standard construction (Section \ref{sec:preliminaries}). It follows from \cite[Theorem 8.4]{LMTEndpA} that $\wt{W}^-=\Lambda_{\text{HM}}^-$ as abstract laminations (the paper proves it for circular pseudo-Anosov flows, but the same method applies to any pseudo-Anosov flow without perfect fits).

Now fix a $\Sigma$-juncture $\beta$. Suppose that $\beta$ is the boundary of a contracting neighborhood of a contracting end $\mathcal{E}$, and $n$ is the smallest positive integer so that $f^n(\mathcal{E})=\mathcal{E}$. The above facts implies that the endpoints of $\partial\wt{\UU}^+_\mu$ can be approximated by endpoints of lifts of $\{f^{-in}(\beta)\}_{i\geq 0}$. The limit set $\mathcal{L}(\wt{\UU}^+_\mu)$ is nowhere dense by the next lemma (Lemma \ref{lem:positive-dense-limit}), so points in $\mathcal{L}(\wt{\UU}^+_\mu)$ are approximated by endpoints of $\partial\mathcal{L}(\wt{\UU}^+_\mu)$. Also note that if $\beta$ is a $\Sigma$-juncture, so is $f^{-in}(\beta)$. Therefore, the endpoints of lifts of any $\Sigma$-juncture in $\wt{\UU}^+_\mu$ are dense inside $\mathcal{L}(\wt{\UU}^+_\mu)$, which completes the proof.
\end{proof}

We give a proof of the following fact used in the proof of Lemma \ref{lem:marker-closure}.

\begin{lem}\label{lem:positive-dense-limit}
The limit set $\mathcal{L}(\wt{\UU}^+_\mu)$ is nowhere dense in $\partial_\infty \wt{L}$.
\end{lem}
\begin{proof}
The subsurface $\UU^+\subset L$ is bounded by leaves of $\FF_L^s$ \cite{LMTEndpA}, each of which has distinct well-defined endpoints at infinity. In particular, $\mathcal{L}(\wt{\UU}^+_\mu)$ is a proper subset of $\partial_\infty \wt{L}$. If $\mathcal{L}(\wt{\UU}^+_\mu)$ contains a nontrivial interval $I$, we can find a hyperbolic element $g\in\pi_1(L)$ with a fixed point in $I$. Here we are using the fact that $L$ has bounded injectivity radius. But $g(\wt{\UU}^+_\mu)$ is either disjoint from or equal to $\wt{\UU}^+_\mu$ since $\UU^+$ is embedded. So $g^n(I)$ is contained in $\mathcal{L}(\wt{\UU}^+_\mu)$ for any integer $n$, a contradiction.
\end{proof}

\begin{lem}\label{lem:dense-rule}
There is a subset $\MM_{\mu,\wt{\Omega}}'$ of $\MM_{\mu,\wt{\Omega}}$ such that
\begin{itemize}
\item[(1)] the set $\xi_\mu(\MM_{\mu,\wt{\Omega}}')$ is dense in $\partial_\infty\mu$;
\item[(2)] the set $\xi_{\wt{\Omega}}(\MM_{\mu,\wt{\Omega}}')$ is dense in $\mathcal{L}(\wt{\UU}^+_\mu)$;
\item[(3)] for any $m\in\MM_{\mu,\wt{\Omega}}'$, we have
\[
\xi_\mu(m)=I_{\wt{\Omega}\mu}(\xi_{\wt{\Omega}}(m)).
\]
\end{itemize}
\end{lem}

\begin{proof}
The subset $\MM_{\mu,\wt{\Omega}}'$ can be taken to be the set of markers in $\MM_{\mu,\wt{\Omega}}$ that arise from Lemma \ref{lem:trivial-holonomy}. Item (3) is exactly the content of Lemma \ref{lem:trivial-holonomy}. Item (1) is true because the endpoints of the lifts of a simple closed curve is dense in $\partial_\infty\mu$. The second part of the proof of Lemma \ref{lem:marker-closure} only used markers in $\MM_{\mu,\wt{\Omega}}'$, so we have already proved item (2).
\end{proof}

For any $\xi\in\partial_\infty\wt{L}$,  we use $c_\xi$ to denote the vertical section of $E_\infty|_{\wt{\Omega}}$ given by $c_\xi(\lambda_t)=(\xi,t)$.

\begin{lem}\label{lem:U-gaps}
For any $\xi\in\partial_\infty(\wt{L})$, the vertical section $c_\xi$ extends continuously to $\mu$ by setting $c_\xi(\mu)=I_{\wt{\Omega}\mu}(\xi)$.
\end{lem}

\begin{proof}
Let $G_\xi$ be the gap of $I_{\wt{\Omega}\mu}$ containing $\xi$, and let $\partial_rG_\xi$ and $\partial_lG_\xi$ be the leftmost and the rightmost endpoint of $G_\xi$ respectively. By Lemma \ref{lem:dense-rule}, we have sequences $\{m_n^+\}$ and $\{m_n^-\}$ in $\MM'_{\mu,\wt{\Omega}}\subset\MM_{\mu,\wt{\Omega}}$ with the properties
\begin{itemize}
\item the points $\xi_{\wt{\Omega}}(m_n^-)$ limits to $\partial_l G_\xi$ from the left and the points $\xi_{\wt{\Omega}}(m_n^+)$ limits to $\partial_r G_\xi$ from the right;
\item we have $\xi_\mu(m_n^\pm)=I_{\wt{\Omega}\mu}(\xi_{\wt{\Omega}}(m_n^\pm))$.
\end{itemize}
These properties imply that the points $\xi_\mu(m_n^-)$ limits to $I_{\wt{\Omega}\mu}(\xi)$ from the left and the points $\xi_\mu(m_n^+)$ limits to $I_{\wt{\Omega}\mu}(\xi)$ from the right by item (3) of Lemma \ref{lem:dense-rule} and Lemma \ref{lem:marker-disjointness}. The two sequences of marker ends pin down the endpoint of $c_\xi$ to be $I_{\wt{\Omega}\mu}(\xi)$ (see the left hand side of \ref{fig:rules}).
\end{proof}

\begin{figure}[!htb]
\labellist
\endlabellist
\includegraphics[scale=.2]{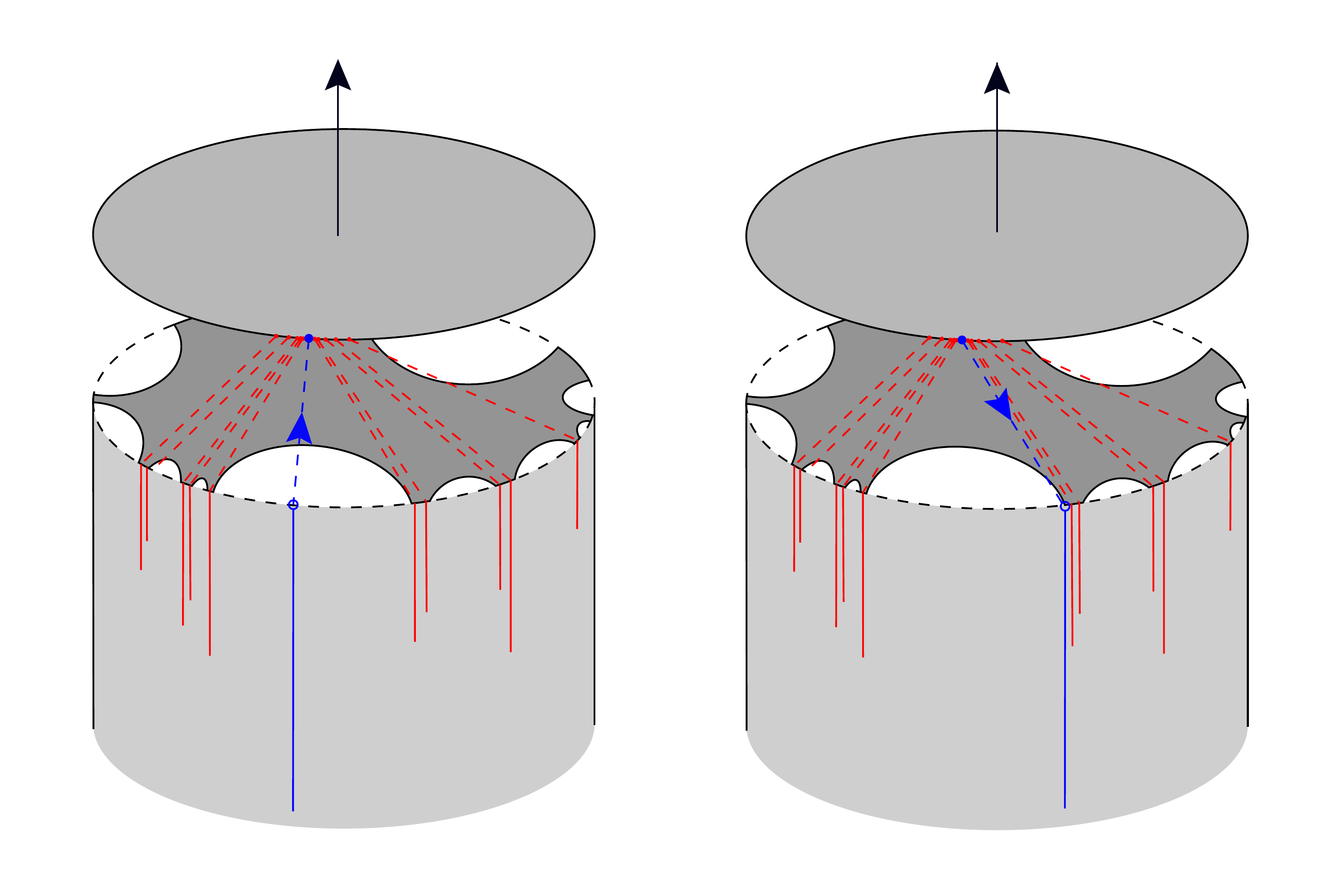}
\caption{The subset $\MM_{\mu,\wt{\Omega}}'$ (red) pins down the way to extend the leftmost sections (blue), with the blue arrows indicating the direction of extension. The lefthand side corresponds to Lemma \ref{lem:U-gaps} and the righthand side corresponds to Lemma \ref{lem:0-to-1}.}\label{fig:rules}
\end{figure}

\begin{lem}\label{lem:0-to-1}
Let $\eta$ be a point in $\partial_\infty\mu$. Suppose $\nu$ is any leaf of $\wt{\FF}$, $\xi$ is any point in $\partial_\infty\nu$ and $s_\xi$ is the leftmost section starting from $\xi$. If the direction of extension of $s_\xi$ at $\mu$ points from $\mu$ to $\wt{\Omega}$, i.e. $\mu$ is closer to $\nu$ in $\Lambda$ than $\wt{\Omega}$, and $s_\xi(\mu)=\eta$, then $s(\lambda_t)$ is the rightmost endpoint of $I_{\wt{\Omega}\mu}^{-1}(\eta)$. Here we view a singleton as a closed interval of length zero.
\end{lem}

\begin{rmk}\label{rmk}
If $\mu$ is in the negative boundary of $\wt{\Omega}$, the lemma remains true if we replace ``the rightmost endpoint" by ``the leftmost endpoint". The proof is the same.
\end{rmk}

\begin{proof}
By Lemma \ref{lem:dense-rule}, there is a sequence of markers $m_n\in\MM'_{\mu,\wt{\Omega}}\subset\MM_{\mu,\wt{\Omega}}$ so that $\xi_\mu(m_n)$ limits to $\eta$ from the right. By the monotonicity of $I_{\wt{\Omega}\mu}$ and item (3) of Lemma \ref{lem:dense-rule}, we have $\xi_{\wt{\Omega}}(m_n)$ limits to the rightmost endpoint of $I_{\wt{\Omega}\mu}^{-1}(\eta)$ from the right. The value of $s_\xi$ on $\wt{\Omega}$ is determined by going down from $\eta$ and following the rightmost-down rule. The lemma can be seen by running the explicit construction. See also the right hand side of Figure \ref{fig:rules}.
\end{proof}

Lemma \ref{lem:U-gaps} and Lemma \ref{lem:0-to-1} give us the complete rules to build the leftmost section starting from a given point. Lemma \ref{lem:U-gaps} tells us how to extend the leftmost section from a product region to an adjacent type-0 leaf, and Lemma \ref{lem:0-to-1} tells us how to go from a type-0 leaf to an adjacent product region.

\section{Building the homeomorphism}\label{sec:homeomorphism}

Suppose $\lambda$ is a leaf of $\FF$ and $\xi$ is a point in $\partial_\infty\lambda$. Let $V_\xi$ denote $I_\lambda^{-1}(\xi)\subset\partial\OO$, which is either a closed interval or a singleton. If $V_\xi$ is a closed interval, let $\partial_lV_\xi$ and $\partial_rV_\xi$ be the leftmost and the rightmost endpoints of $V_\xi$ respectively. Note by our convention, left means clockwise and right means counterclockwise. We say $V_\xi$ is a stable (resp. unstable) gap of $\lambda$ if $Q_\lambda^{-1}(\xi)$ is a boundary $\FF_\OO^{s}$-leaf (resp. $\FF_\OO^{u}$-leaf) of $p(\lambda)$. Note that a closed interval of $\partial\OO$ cannot be a stable gap of a leaf while being an unstable gap of another. Indeed, a leaf of $\FF_\OO^s$ and a leaf of $\FF_\OO^u$ cannot bound an ideal bigon because $\phi$ has no perfect fits. We say a closed interval of $\partial\OO$ is a stable (resp. an unstable) gap if it is a stable (resp. an unstable) gap of a leaf of $\wt{\FF}$.

Since $\Su_{\mathrm{left}}$ is the completion of the set $LS$ of leftmost sections, to define the homeomorphism $T:\Su_{\mathrm{left}}\to\partial\OO$, it suffices to define $T$ on $LS$ and show that it admits an extension. For any point $\xi$ in $E_\infty$, let $s_\xi$ be the leftmost section starting from $\xi$. The set $LS^*$ of \textbf{pointed leftmost sections} is the set
\[
\{(s,\xi)|~s\in LS, \xi\in E_\infty, s=s_\xi\}.
\]
We point out that this definition is not redundant, for it is possible to have $s_\xi=s_{\xi'}$ for $\xi\neq\xi'$. There is a natural forgetful map $\pi:LS^*\to LS$ defined by forgetting the starting point. We define a map $T^*:LS^*\to\partial\OO$ as follows. For a pointed leftmost section $(s,\xi)$, consider $V_\xi$. If $V_\xi$ is a single point, define $T^*(s,\xi)=V_\xi$. If $V_\xi$ is an unstable gap, define $T^*(s,\xi)=\partial_lV_\xi$. If $V_\xi$ is a stable gap, define $T^*(s,\xi)=\partial_rV_\xi$.

Theorem \ref{thm:main} follows from the following theorem.

\begin{thm}\label{thm:homeo-of-circles}
    The map $T^*$ descends to a map $T:LS\to\partial\OO$ and extends continuously to a map from $\Su_\mathrm{left}$ to $\partial\OO$, which we will again denote by $T$. The extension $T$ is injective and preserves the cyclic order, hence a homeomorphism. Moreover, $T$ is $\pi_1(M)$-equivariant, and for any $\lambda\in\Lambda$ and any $s\in\Su_\mathrm{left}$, we have $U_\lambda(s)=I_\lambda(T(s))$. In other words, $T$ is an isomorphism of universal circles between $\Su_\mathrm{left}$ and $\partial_\OO$.
\end{thm}
 
The rest of this section will be contributed to proving Theorem \ref{thm:homeo-of-circles}.

\begin{lem}\label{lem:well-definedness}
    Suppose $(s,\xi_0)$ is an element in $LS^*$. Let $\mu$ be any leaf of $\widetilde{\FF}$ and suppose $s(\mu)=\xi$. Then $T^*(s,\xi_0)\in V_\xi$.
\end{lem}

\begin{proof}
    By definition we have $s=s_{\xi_0}$ is the leftmost section starting from $\xi_0$, where $\xi_0\in\partial_\infty\lambda_0$ for some $\lambda_0\in\Lambda$. We assume that $\lambda_0$ is a type-1 leaf for simplicity. The case when $\lambda_0$ is type-0 is basically the same.

    We first consider the case when $\lambda_0$ and $\mu$ are comparable and $\lambda_0<\mu$. Take a sequence of $\wt{\FF}$-leaves
    \[
    \lambda_0\to\lambda_1\to\lambda_2\to\cdots\to\lambda_n=\mu
    \]
where $\lambda_{2i+1}$ is a type-0 leaf, $\lambda_{2i}$ is a type-1 leaf and $\lambda_{i}\lesssim\lambda_{i+1}$. The sequence represents the shortest path from $\lambda_0$ to $\mu$ in $\Lambda^*$, viewing a type-1 leaf as the vertex representing the corresponding product region. We record the value of $s$ along this sequence by $s_i=s(\lambda_i)$ and let $V_i=V_{s(\lambda_i)}$. We have a sequence of closed intervals (possibly with length zero) $V_i\subset\partial\OO$. The goal is to show that for all $\lambda_i$, $T^*(s,\xi_0)\in V_i$. In particular, this will implies $T^*(s,\xi_0)\in V_n=V_\xi$. We will show this by tracking how the intervals $V_i$ vary along the sequence $\lambda_i$. By Lemma \ref{lem:U-gaps}, we have $V_{2i}\subset V_{2i+1}$; by Lemma \ref{lem:0-to-1}, we have $V_{2i+1}\subset V_{2i+2}$.

\begin{lem}\label{lem:entering-exiting-gap}
If $V_{2i}\subsetneq V_{2i+1}$, then $V_{2i+1}$ is a stable gap. If $V_{2i+2}\subsetneq V_{2i+1}$, then $V_{2i+1}$ is an unstable gap and $V_{2i+2}=\partial_l V_{2i+1}$.
\end{lem}

\begin{proof}
First, suppose $V_{2i}\subsetneq V_{2i+1}$. The leaf $e:=Q_{\lambda_{2i+1}}^{-1}\left(I_{\lambda_{2i+1}}(V_{2i+1})\right)$ is a side of $p(\lambda_{2i+1})$ containing in the interior of $p(\lambda_{2i})$. The fixed point $x_e$ in $e$ under $\mathrm{Stab}(e)$ corresponds to a periodic $\wt{\phi}$-orbit $\gamma$ in $\wt{M}$ intersecting $\lambda_{2i}$ but not intersecting $\lambda_{2i+1}$. Since every $\wt{\phi}$-orbit in $\wt{\FF}^s(\phi)$ is forward asymptotic to $\gamma$, we see that $\FF^s_\OO(x)$ is not contained in $p(\lambda_{2i+1})$. This means $e=\FF_\OO^s(x)$ and $V_{2i+1}$ is a stable gap.

Now suppose $V_{2i+1}\subsetneq V_{2i+2}$. A similar argument as above shows $V_{2i+1}$ is an unstable gap. By Lemma \ref{lem:0-to-1} and Remark \ref{rmk}, if $V_{2i+2}$ is a closed interval of positive length, there will be a boundary leaf $e$ of $p(\lambda_{2i+1})$ and a boundary leaf $e'$ of $p(\lambda_{2i+2})$ different from $e$ and sharing the leftmost endpoint $\chi$ with $e$. But this contradicts Lemma \ref{lem:disjoint-edges}. Therefore, $V_{2i+2}$ is a single point, and it is the leftmost endpoint of $V_{2i+1}$ by Lemma \ref{lem:0-to-1}.
\end{proof}

We continue the proof that $T^*(s,\xi_0)\in V_i$ for all $0\leq i\leq n$. It is obvious that $T^*(s,\xi_0)\in V_0$ by the definition of $T$. If $V_0$ is a single point or a stable gap of $\lambda_0$, by Lemma \ref{lem:entering-exiting-gap} we have $V_i\subset V_{i+1}$ for all $i$. So we have $T^*(s,\xi_0)\in V_i$.

If $V_0$ is an unstable gap, then there are two cases. If for all $0\leq i\leq n$ we have
$V_i=V_0$
then there is nothing to prove. If this is not the case, let $N$ be the first positive integer so that $V_N\neq V_0$. If $V_0\subsetneq V_N$, then $V_N$ is a stable gap by Lemma  \ref{lem:entering-exiting-gap}, and we again have $V_i\subset V_{i+1}$ for all $i\geq N-1$. If $V_N\subsetneq V_0$, then $V_N=T^*(s,\xi_0)=\partial_l V_0$ by definition and Lemma \ref{lem:entering-exiting-gap}. We use Lemma \ref{lem:entering-exiting-gap} again to see that $\{V_i\}_{i\geq N}$ is a monotone increasing sequence of closed intervals. In any case, we have $T^*(s,\xi_0)\in V_i$ for all $i$.

The case when $\mu<\lambda_0$ can be proved using a similar argument. Thus the lemma is proved for $\mu$ comparable to $\lambda_0$.

    Now suppose $\mu$ is not comparable to $\lambda_0$. We again consider the shortest path from $\lambda_0$ to $\mu$ in $\Lambda$, and track how $V_i$ changes along the path. To illustrate the idea, we consider the following example. Suppose the shortest path from $\lambda$ to $\mu$ in $\Lambda^*$ is of length five:
    \[
    \lambda_0\to\lambda_1\to\cdots\to\lambda_4=\mu
    \]
where $\lambda_0,\lambda_2,\lambda_4$ are type-1 leaves, $\lambda_1,\lambda_3$ are type-0 leaves, and they satisfy $\lambda_0\lesssim\lambda_1\lesssim\lambda_2$ and $\lambda_4\lesssim\lambda_3\lesssim\lambda_2$. Let $V_{i}=V_{s(\lambda_i)}$. We made a turn at $\lambda_2$ from the positive flow direction to the negative direction. Our previous discussion shows that $T^*(s,\xi_0)\subset V_i$ for $i=0,1,2$. Since $\lambda_1$ and $\lambda_3$ are incomparable, the core of $I_{\lambda_1}$ is contained in a single unstable gap $G$ of $I_{\lambda_3}$. The interval $V_3$ must be the gap $G$ containing $V_2$ by Lemma \ref{lem:U-gaps}, Lemma \ref{lem:0-to-1} and the definition of $T$. In particular, we have $T^*(s,\xi_0)\subset V_3$. Since $\lambda_3$ is negatively adjacent to $\lambda_2$ and $V_2\subsetneq V_3$, a similar argument to the proof of Lemma \ref{lem:entering-exiting-gap} shows that $V_3$ is an unstable gap, and an unstable gap will only become larger as we track $V_i$ backwards. Therefore, we have $T^*(s,\xi_0)\in V_3\subset V_4$.

In general, the path from $\lambda$ to $\mu$ has a finite number of turns. If we track the interval $V_i$ along the path, at a turn from the positive direction to the negative direction, $V_i$ will become an unstable gap and can only grow larger until the next turn happens. Similarly, if we turn from the negative direction to the positive direction, $V_i$ will become a stable gap and can only grow larger until the next turn happens. Hence, $V_i$ will be non-decreasing after we made the first turn. But we have show that $T^*(s,\xi_0)$ is in $V_i$ before we made any turn in the first part of the proof. So the proof is complete.
\end{proof}

\begin{cor}\label{cor:intersection}
    Let $(s,\xi_0)\in LS^*$ be a pointed leftmost section. Then we have $T^*(s,\xi_0)=\bigcap_{\lambda\in\Lambda}V_{s(\lambda)}$.
\end{cor}

\begin{proof}
Lemma \ref{lem:well-definedness} already shows that $T^*(s,\xi_0)\in\bigcap_{\lambda\in\wt{\FF}}V_{s(\lambda)}$, so it suffices to prove that there is some $\lambda$ with $V_{s(\lambda)}$ being a singleton. Suppose $\xi_0$ is at the infinity of the leaf $\lambda_0$. If $V_{\xi_0}$ is a single point, it is trivial. If $V_{\xi_0}$ is a nontrivial closed interval, let $e$ be the side of $p(\lambda_0)$ facing $V_{\xi_0}$ and let $x_e$ be the periodic point in $e$ as in Proposition \ref{prop:local-dynamics}. The periodic orbit $p^{-1}(x_e)$ intersects some leaf $\lambda$ comparable to $\lambda_0$. If we take a path in $\Lambda^*$ from $\lambda_0$ to $\lambda$ and record the intervals $V_i$ as in the proof of Lemma \ref{lem:well-definedness}, there must be some $i$ so that $V_{i+1}\subsetneq V_i$. Otherwise we have $V_0\subset V_1\subset\cdots$ and so $I_\lambda$ has a gap containing $V_{\xi_0}$, but that contradicts $x_e\in p(\lambda)$. By Lemma \ref{lem:entering-exiting-gap}, $V_{i+1}$ is a single point, so is the intersection $\bigcap_{\lambda\in\wt{\FF}}V_{s(\lambda)}$. The lemma is proved.
\end{proof}

\begin{cor}\label{cor:pointed-descend}
There is a map $T:LS\to\partial\OO$ so that the following diagram commutes:
\[
\xymatrix{
LS^*\ar[dr]_{T^*}\ar[r]^{\pi} & LS\ar[d]^{T}\\
 & \partial\OO
}
\]
\end{cor}
\begin{proof}
For any $s\in LS$, pick a starting point $\xi$ for $s$ and define $T(s)=T^*(s,\xi)$. By Corollary \ref{cor:intersection}, we have $T^*(s,\xi)=T^*(s,\xi')$ for any $(s,\xi),(s,\xi')\in LS^*$. Therefore, the map $T$ is well-defined.
\end{proof}

\begin{cor}\label{cor:cyclic-order-preserving}
    The map $T$ preserves the cyclic order of elements in $LS$.
\end{cor}

\begin{proof}
    The cyclic order of leftmost sections are determined by the cyclic order of their values on embedded lines in the leaf space (\cite[Lemma 6.25]{Calegari:2003aa}, see also Section \ref{sec:universal-circle}). Their images under $T$ must follow the same cyclic order by Lemma \ref{lem:well-definedness}.
\end{proof}

\begin{lem}\label{lem:dense-image}
    The image of $LS$ under $T$ is dense.
\end{lem}

\begin{proof}
    It is clear that $T$ is $\pi_1(M)$-equivariant, so $T(LS)$ is a $\pi_1(M)$-invariant subset of $\partial\OO$. The lemma follows from the minimality of the $\pi_1(M)$-action on $\partial\OO$ \cite{Fenley2009}.
\end{proof}

\begin{lem}\label{lem:continuous-extension}
    The map $T:LS\to\partial\OO$ extends continuously to a homeomorphism $T:\Su_\mathrm{left}\to\partial\OO$.
\end{lem}

\begin{proof}
By Corollary \ref{cor:cyclic-order-preserving}, Lemma \ref{lem:dense-image} and the fact that $\Su_\mathrm{left}$ is the completion of $\QQ$, there is a unique continuous extension of $T$ to $\Su_\mathrm{left}$, and the extended $T$ is a homeomorphism between $\Su_\mathrm{left}$ and $\partial\OO$.
\end{proof}

\begin{proof}[Proof of Theorem \ref{thm:homeo-of-circles}]
It suffices to show the ``moreover'' part about the map $T$ defined above. The $\pi_1(M)$-equivariance is automatic from the way we define $T$. The structure maps are intertwined by $T$ because of Lemma \ref{lem:well-definedness}.
\end{proof}

It is also possible to define the universal circle from rightmost sections $\Su_{\mathrm{right}}$, by considering the completion of rightmost sections (i.e. the sections of $E_\infty$ that go rightmost-up and leftmost-down). In general, there is no reason to expect $\Su_\mathrm{left}=\Su_\mathrm{right}$. However, we have the following corollary of Theorem \ref{thm:main}.

\begin{cor}
Under the assumption of Theorem \ref{thm:main}, the universal circles $\Su_\mathrm{left}$ and $\Su_\mathrm{right}$ are isomorphic.
\end{cor}

\begin{proof}
Using the same proof of Theorem \ref{thm:main}, it can be shown that $\Su_\mathrm{right}$ is isomorphic to $\partial\OO$, hence isomorphic to $\Su_\mathrm{left}$.
\end{proof}

\section{Invariant laminations}\label{sec:invariant-laminations}

We conclude with a discussion of the invariant laminations on $\partial\OO\cong\Su_\mathrm{left}$. See Section \ref{sec:preliminaries} for a discussion about laminations on a circle.

Any $\lambda\in\Lambda$ separates $\Lambda$ into two components, the one $\Lambda^+(\lambda)$ containing the flow positive side of $\lambda$ and the one $\Lambda^-(\lambda)$ containing the flow negative side of $\lambda$. The leaf $\lambda$ also separates $\wt{M}$ into two parts, denoted by $\wt{M}^+(\lambda)$ and $\wt{M}^-(\lambda)$ with the same sign convention. Define a subset $\Xi^\pm(\lambda)$ of $\mathrm{Symm}_2(\partial\OO)$ on by
\[
\Xi^\pm(\lambda)=\partial\mathrm{CH}\big(\bigcup_{\mu\in\Lambda^\pm(\lambda)}\mathrm{core}(I_\mu)\big).
\]
The set $\Xi^\pm$ are then defined as
\[
\Xi^\pm=\overline{\bigcup_{\lambda\in\Lambda}\Xi^\pm(\lambda)}.
\]
It is proved in \cite{Calegari2002PromotingEL} that $\Xi^\pm$ is indeed a pair of $\pi_1(M)$-invariant laminations on $\partial\OO$.

We are now ready to prove Theorem \ref{thm:invariant-laminations}. Recall that $\mathcal{L}_\OO^{s/u}$ is the lamination on $\partial\OO$ induced by $\FF_\OO^{s/u}$. 

\begin{proof}[Proof of Theorem \ref{thm:invariant-laminations}]
Take any type-0 leaf $\lambda$ and consider the shadow $p(\lambda)$. Since $\lambda$ is not a fiber, \cite[Proposition 4.6]{Fenley1999823} shows that $\partial p(\lambda)$ has both leaves in $\FF_\OO^u$ and $\FF_\OO^s$. Suppose that $e$ is a side of $p(\lambda)$ that is contained in a leaf of $\FF_\OO^s$ and consider the leftmost section $s$ starting from $Q_\lambda(e)$. If $\mu$ is a leaf in $\Lambda^+(\lambda)$, we take a path in $\Lambda^*$ from $\lambda$ to $\mu$ and track the closed interval $V_i$ as in the proof of Lemma \ref{lem:well-definedness}. The interval $V_0$ is the stable gap of $I_\lambda$ facing $e$, and the proof of Lemma shows \ref{lem:well-definedness} that $V_i$ is monotone increasing along the path. This means the shadow of $\mu$ is on the same side of $e$ as $p(\lambda)$. The side $e$, viewed as an element of $\mathrm{Symm}_2(\partial\OO)$, is therefore in $\Xi^+(\lambda)$, hence in $\Xi^+$. By transitivity of $\phi$, every leaf of $\FF^s$ is dense in $M$. This implies that the $\pi_1(M)$-image of $\partial e$ is dense in $\mathcal{L}_\OO^s$ because $\phi$ has no perfect fits. Since both $\Xi^+$ and $\FF_\OO^s$ are $\pi_1(M)$-invariant and closed, we have $\Xi^+\supset\mathcal{L}_\OO^s$.

If $\Xi^+-\mathcal{L}_\OO^s$ is not empty, the difference must be a union of diagonals of complementary regions of $\mathcal{L}_\OO^s$. Note that these diagonals cannot be approximated by leaves in $\mathcal{L}_\OO^s\cup\mathcal{L}_\OO^u$, so there must be such a diagonal $d$ in $\Xi^+(\lambda)$ for some $\lambda$. The corresponding complementary polygon comes from a singular leaf $l$ of $\FF_\OO^s$, and we denote the singularity in $l$ by $s$.

Suppose $d$ has endpoints $\xi$ and $\xi'$. Then there is a sequence of leaves $\lambda_n\in\Lambda^+(\lambda)$, sides $e_n$ of $p(\lambda_n)$ and endpoints $\xi_n$ of $e_n$ so that $\xi_n$ converges to $\xi$. Up to taking a subsequence, we can assume that all $e_n$ are contained in leaves of $\FF_\OO^s$ or leaves of $\FF_\OO^u$. If all $e_n$ are contained in leaves of $\FF_\OO^u$, then $d$ cannot be a boundary component of the convex hull of $\bigcup_{\mu\in\Lambda^\pm(\lambda)}\mathrm{core}(I_\mu)$, because $e_n$ will eventually crosses $d$ by the absence of perfect fits. So all $e_n$ are contained in leaves of $\FF_\OO^s$. In particular, the singularity $s$ can be approximated by points in shadows of leaves in $\Lambda^+(\lambda)$. We will show that this is impossible, a contradiction.

If $s$ is in $p(\lambda)$, then there are points of $\mathrm{core}(I_\lambda)$ on both sides of $d$, contradicting the assumption that $d\in\Xi^+(\lambda)$. Therefore, the $\wt{\phi}$-orbit $p^{-1}(s)$ is disjoint from $\lambda$. If $p^{-1}(s)$ is contained in $\wt{M}^+(\lambda)$, it contradicts our assumption that $d\in\Xi^+(\lambda)$ by a similar reason. So $p^{-1}(s)$ is contained in $\wt{M}^-(\lambda)$. Note that $s$ is not in $\partial p(\lambda)$, otherwise a face of $l$ will be a side of $p(\lambda)$, and the face will be a leaf of $\Xi^+(\lambda)$ as we showed above. This again contradict the assumption that $d\in\Xi^+(\lambda)$. So orbits close enough to $p^{-1}(s)$ will stay in $\wt{M}^-(\lambda)$, giving the desired contradiction.

Therefore, $\Xi^+$ must be the same as $\mathcal{L}_\OO^s$. By the same reason we have $\Xi^-=\FF_\OO^u$, finishing the proof of Theorem \ref{thm:invariant-laminations}.
\end{proof}

\bibliographystyle{myplain}
\bibliography{refs}

\end{document}